\documentclass[leqno, 11pt, a4paper]{ip-journal}
\usepackage{amssymb}

\usepackage{tikz}

\usepackage{graphicx,amssymb,amsfonts,epsfig,amsthm,a4,amsmath,url}
\usepackage[latin1]{inputenc}

\usepackage{latexsym}
\usepackage{amsmath,amsthm}
\usepackage{amsfonts}
\usepackage{epsfig}
\usepackage{amssymb}
\usepackage{amscd}
\usepackage[all]{xy}
\usepackage{stmaryrd}

\input xy
\usepackage{hyperref}

\usepackage{enumitem}

\newtheorem{Thm}{Theorem}[section]
\newtheorem{Prop}[Thm]{Proposition}
\newtheorem{Lem}[Thm]{Lemma}

\theoremstyle{definition}
\newtheorem{Def}[Thm]{Definition}

\newcommand{\Z}{\mathbf{Z}}

\newcommand{\C}{\mathbf{C}}

\newcommand{\bpr}{\noindent \textbf{Proof}: }

\newcommand{\epr}{~$\blacksquare$}

\newcommand{\ind}{\mathrm{ind}}

\newcommand{\tr}{\mathrm{tr}}
\newcommand{\cA}{\mathcal{A}}

\numberwithin{equation}{section}

\title{Higher rho invariant and delocalized eta invariant at infinity}
	\author{Xiaoman Chen}
\address[Xiaoman Chen]{ School of Mathematical Sciences, Fudan University}
\email{xchen@fudan.edu.cn}
\thanks{The first author is partially supported by NSFC 11420101001.}

\author{Hongzhi Liu}
\address[Hongzhi Liu]{School of Mathematics, Shanghai University of Finance and Economics}
\email{liu.hongzhi@sufe.edu.cn}
\thanks{The second author is partially supported by NSFC 11901374.}

\author{Hang Wang}
\address[Hang Wang]{ School of Mathematical Sciences, East China Normal University}
\email{wanghang@math.ecnu.edu.cn}
\thanks{The third author is partially supported by the Shanghai Rising-Star Program grant 19QA1403200, and by NSFC 11801178.}

\author{Guoliang Yu}
\address[Guoliang Yu]{ Department of Mathematics, Texas A\&M University}
\email{guoliangyu@math.tamu.edu}
\thanks{The fourth author is partially supported by NSF 1700021, NSF 1564398, and Simons Fellows Program}

\date{\today}

\begin{document}
	
	\baselineskip=16pt
	
		\begin{abstract}
		
%		In this paper, we define a pair of secondary invariants at infinity, higher rho invariant at infinity and delocalized eta invariant at infinity, for the Dirac operator on a complete manifold with metric with a uniform positive scalar curvature metric outside a compact set. In this case, we establish a generalizaiton of the Atiyah-Patodi-Singer index theorem. We also apply our invariants to study the geometry of a manifold with corner with positive scalar curvature metric on each boundary face.

In this paper, we introduce several new secondary invariants for Dirac operators on a complete Riemannian manifold with a uniform positive scalar curvature metric outside a compact set and use these secondary invariants to establish a higher index theorem for the Dirac operators.
We apply our theory to study the secondary invariants for a manifold with corner with positive scalar curvature metric on each boundary face.

	\end{abstract}
	
	\maketitle

\bibliographystyle{plain}

	\tableofcontents

	\section{Introduction}
	
	In this article, we introduce a new theory of secondary invariants for Dirac operators on noncompact spin manifolds endowed with metrics with uniform positive scalar curvature at infinity, i.e. higher rho invariants and the delocalized eta invariants. Let $X$ be a complete spin manifold with metric admitting uniform positive scalar curvature outside a compact set $Z$, $D$ be the Dirac operator on the universal covering space $\widetilde{X}$ of $X$, and let $G$ be $\pi_1 (X)$, the fundamental group of $X$. We obtain a higher index formula for the Dirac operator $D$ which expresses the delocalized trace of the higher index in terms of delocalized secondary invariants at infinity, as follows:
	
		\begin{Thm}\label{thm in intr main 1}
	Let $D$ be the Dirac operator on the universal covering space $\widetilde{X}$ of a spin manifold $X$ with uniform positive scalar curvature outside a compact subset $Z$ defined as above. Let $G$ be the fundamental group of $X$. Let $g\in G$ be a nontrivial element whose conjugacy class has polynomial growth. Let $ind_G(D)$ be the higher index of the Dirac operator $D$. The delocalized trace at $g$ of the higher index of $D$, $\tr_{g}(\ind_G (D))$ equals the half of the negative delocalized eta invariant at infinity, $\eta_{g, \infty}(D)$, i.e.
	 \begin{equation}\label{eq:main eq 1 in intro}
	 \tr_{g}(\ind_G (D))=-\frac{1}{2}\eta_{g, \infty}(D).
	 \end{equation}
	 Furthermore, we have
	 	\begin{equation}\label{eq:main eq 2 in intro}
	 \frac{1}{2}\eta_{g, \infty}(D)=\lim_{t\to 0}\int_t^{\infty}\tr_g(e^{-sD_c^-D_c^+}D_c^- [D^+,\psi_2])ds.
	 \end{equation}
	 where $D_c$ is the invertible Dirac operator on $\widetilde X\backslash \widetilde Z$, and $\psi_2$ is a $G$-invariant cutoff function from $\widetilde{X}$ to $[0,1]$, which equals $0$ on a cocompact neighbourhood of $\widetilde Z$, the $G$-Galois covering space of $Z$, and equals $1$ outside a cocompact set larger than $Z$. The  integral on the right hand side of \eqref{eq:main eq 2 in intro} is independent of the choice of the cutoff function $\psi_2.$
	\end{Thm}

     The higher index $\ind_G(D)$ of $D$, on a complete manifold with invertibility condition at infinity,  was first introduced in Bunke (cf. \cite{Bunke95}). The main significance and subtlety of the higher index formula in Theorem \ref{thm in intr main 1} is that $\tr_g(\ind_G(D))$ vanishes when $X$ is closed (cf. \cite{weinbyu15finitepart}).  Furthermore, when $X$ is a cylindrical manifold obtained from a compact manifold with boundary $\partial M$, and with a metric with positive scalar curvature on the boundary which is collared near $\partial M$, then $\tr_g(\ind_G(D))$ equals the delocalized eta invariant of the Dirac operator on $\partial M$.
	
	 Equation \eqref{eq:main eq 1 in intro} is obtained by considering the higher rho invariant at infinity. These secondary invariants we introduce demonstrate the nonlocality of the higher index of $D$. We also apply the higher rho invariant at infinity to analyze the geometry of manifolds with corner of codimension 2 with positive scalar curvature metrics on all boundary faces.
	
	The classical eta invariant is a nonlocal spectral invariant of Dirac type operators (cf. \cite{APSspectral0,APSspectral1,APSspectral2,APSspectral3}). Let $M$ be a closed odd dimensional spin manifold with a positive scalar curvature metric, $D_M$ be the Dirac operator on $M$. When the following integral converges, the eta invariant of $D_M$ can be expressed as
	\[
	\frac{2}{\sqrt{\pi}} \int_0^\infty \tr(e^{-t^2D^2_{M}} D_{M}) dt
	\]
	where $\tr$ is the operator trace. To take into account of the information of the fundamental group, Lott introduced the delocalized eta invariant (cf. \cite{Lott99}). More precisely, let $G$ be the fundamental group of $M$, $\widetilde M$ be the universal covering space of $M$, and $\widetilde D_{M}$ be the lifting of $D_M$ to $\widetilde M$. Let $g$ be a nontrivial element of $G$, whose conjugacy class $\langle g \rangle$ has polynomial growth. The delocalized eta invariant of $\widetilde D_{M}$ at $g$, $\eta_g (\widetilde D_{M})$, introduced by Lott in \cite{Lott99}, is defined by the following integration
	\[
	\frac{2}{\sqrt{\pi}}\int_0^\infty \tr_g (e^{-t^2\widetilde{D}^2_{ M}} \widetilde D_{M}) dt.
	\]
	Here $\tr_g$ is the following trace map
	\[
	\tr_g (A) = \sum\limits_{h\in \langle g \rangle} \int_{\mathcal F}A(x,hx) dx,
	\]
	on $G$-equivariant Schwartz kernels $A\in C^\infty (\widetilde M \times \widetilde M)$,
	where $\mathcal F$ is a fundamental domain of $\widetilde M$ under the $G$-action. One can also define $\tr_g$ for continuous group, see \cite{HWW19equivariantAPSindexforproperaction} for example. There is also a higher generalization of the pairing between $\tr_g$ and $K$-theory of geometric $C^*$-algebras, introduced by \cite{PflaumPosthumaTang2015}, \cite{ChenWangXieYu}, \cite{PiazzaSchichZenobi}, and \cite{SongTang19}, which is to consider cyclic cocycles.
	
	 Since the metric $m$ on $M$ admits positive scalar curvature, it follows from the Lichnerowicz formula that the higher index of $D_M$, $\ind_G (\widetilde D_{M})$, in the $K$-theory of the group $C^*$-algebra $C^*_r(G)$, is trivial with a specific trivialization. In this case, Higson and Roe proposed to study a secondary invariant in $K$-theory of a certain $C^*$-algebra, the higher rho invariant of $\widetilde D_{M}$, $\rho(\widetilde D_{M})$ (cf. \cite{Roe96, HigsonRoe10}). The higher rho invariant of the Dirac operator has been applied to estimate the lower bound of how many positive scalar curvature metrics a manifold can bear (cf. \cite{Piazzaschick14, xieyu14arelative, XieandYu14positivescalarcurvature, XieYu17higherandmoduli, weinbyu15finitepart}). The higher rho invariant is closely related to the delocalized eta invariant. Let $g$ be a nontrivial element of $G$, whose conjugacy class $\langle g \rangle$ has polynomial growth. In \cite{XYdelocalizedetainvalgebraicityandKtheoryofgroupCalgebras}, Xie and Yu defined a canonical determinant map $\tau_g$ associated to $g$, and showed that $\tau_g (\rho(\widetilde D_{M})) = \frac{1}{2}\eta_g (\widetilde D_{M})$.

	The eta invariants and the higher rho invariants appear in the study on the geometry of manifolds with boundary naturally.
	Let $M$ be an even dimensional spin manifold with boundary $N$, with the metric $m$ having product structure near $N$, and admitting positive scalar curvature when restricted  to $N$. Then the eta invariant of the Dirac operator $D_{N}$ on the boundary, $\eta(D_{N})$, is the the correction term in the formula of the Fredholm index of the Dirac operator $D_{M_\infty}$ on $M_\infty:= M\cup \partial M \times [0, \infty) $. Let $\text{index}(D_{M_\infty})$ be the Fredholm index of  $D_{M_\infty}$. Then
	\begin{equation}\label{eq: aps thm}
		\text{index} (D_{M_\infty}) =\int_M \hat A(M) - \frac{\eta(D_{N})}{2},
	\end{equation}
	where $\hat{A}(M)$ is the $\hat{A}$ genus (cf. \cite{APSspectral0, APSspectral1, APSspectral2, APSspectral3}). The eta invariant explains the nonlocality of the Fredholm index. Equation \ref{eq: aps thm} is often referred to as the Atiyah-Patodi-Singer index theorem. On the other hand, the delocalized eta invariant and the higher rho invariant capture the nonlocality of the higher index of $D_{M_\infty}$. In fact, in \cite{XYdelocalizedetainvalgebraicityandKtheoryofgroupCalgebras,ChenWangXieYu}, Xie-Yu and Chen-Wang-Xie-Yu obtained several higher generalizations of the Atiyah-Patodi-Singer index theorem. In particular, for any nontrivial element $g\in G=\pi_1(M)$ whose conjugacy class is of polynomial growth, they proved
	\begin{equation}\label{eq: higher aps thm}
		\tr_g(\ind_G(\widetilde D_{M_\infty})) =-\tau_g (\rho(\widetilde D_{N})) =-\frac{1}{2}\eta_g(\widetilde D_{N}).
	\end{equation}

We mention that the formula
\[
\tr_g(\ind_G(\widetilde D_{M_\infty}))  =-\frac{1}{2}\eta_g(\widetilde D_{N})
\]
was first established in \cite{PaoloThomas2007jncg} by Piazza and Schick, where their proof employs as a crucial ingredient the specialization of $0$-degree forms of the higher Atiyah-Patodi-Singer index formula of Leichtnam and Piazza (\cite{LeichtPiazza1997}, \cite{Piazzaschick14}). And the formula
\[
\tau_g (\rho(\widetilde D_{N})) =\frac{1}{2}\eta_g(\widetilde D_{N})
\]
was also obtained by Piazza, Schick, and Zenobi in \cite{PiazzaSchichZenobi}.
	
	Since the metric on $M_\infty$ has positive scalar curvature outside the compact set $M$, the Dirac operator $D_{M_\infty}$ is invertible at infinity. In this particular case, the delocalized eta invariant and the higher rho invariant of $D_N$ can be viewed as secondary invariants at infinity associated to $D_{M_\infty}$. The study of the secondary invariants  has led to several major breakthroughs to the estimate of the lower bound of the rank of the abelian group formed by the concordance classes of positive scalar curvature metric (cf. \cite{Piazzaschick14,XieandYu14positivescalarcurvature,xieyu14arelative,XieYu17higherandmoduli,XieYuRudolf}). In the meanwhile, a parallel study to the secondary invariants allows one to estimate the lower bound of the topological structure group (\cite{WXY16}).
	
%	The metric on $M_\infty $ is a special case of ``positive at infinity".
	 The first main result, Theorem \ref{thm in intr main 1}, is a generalization of the above Atiyah-Patodi-Singer index theorem and its higher counterpart to the  case of a noncompact complete manifold with uniform positive scalar curvature metric at infinity.
	
	Let $X$ be an $n$-dimensional complete spin manifold and $Z$ be a compact subset of $X$, let $m$ be a metric on $X$, which has uniformly positive scalar curvature outside $Z$, and $G$ be the fundamental group of $X$. Let $\widetilde{X}$ be the universal covering space of $X$ and $\widetilde{Z}$ be the induced $G$-Galois covering space of $Z$. In this case, following  \cite{Roe2016positivecurvaturepartialvanishing} and \cite{XieandYu14positivescalarcurvature}, one can define the higher index of the Dirac operator $D$ on ${ \widetilde X}$ in $K_n(C^*_r(\widetilde Z)^G)$, which is isomorphic to $K_n(C^*_r(G))$. See Subsections \ref{sec:higher index a definition} and \ref{subsec: higher index different approach} for details. The higher index of $D$ is denoted as $\ind_G(D)$ or $\ind_{ \widetilde X,\widetilde Z,T}(D)$, where $T$ is any sufficiently large number. We emphasize that the definition of the higher index of $D$ is independent of $T$. However, the subscript $T$ manifests the particular choice of a representative class of $\ind_G(D)$. In Subsection \ref{subsec higher rho at infinity}, the higher rho invariant at infinity of $D$, denoted as $\rho_{ \widetilde X,\widetilde Z,T}(D)$, is defined as the image of $\ind_{ \widetilde X,\widetilde Z,T}(D)$ under the connecting map of a $K$-theoretic six-term exact sequence associated to a canonical short exact sequence of geometric $C^*$-algebras. Moreover, for nontrivial element $g\in G$, whose conjugacy class is of polynomial growth, the delocalized eta invariant at infinity of $D$ is simply defined to be
	\begin{equation}\label{line intro tau rho in intro}
	2\tau_g (\rho_{ \widetilde X,\widetilde Z,T}(D)).
	\end{equation}
	As in the case of manifold with cylindrical end, the number in line \eqref{line intro tau rho in intro} is essential for us to obtain Equation \eqref{eq:main eq 1 in intro} in Theorem \ref{thm in intr main 1}.

	In the meanwhile, Equation \eqref{eq:main eq 2 in intro} in Theorem \ref{thm in intr main 1} is established by applying the method in \cite{HWW19equivariantAPSindexforproperaction}, which is invented when Hochs, Wang and Wang was to develop a new approach to obtain a refinement of the Atiyah-Patodi-Singer index theorem. See Section \ref{subsec:formulafordelocalizedeta} for details. At the end of the paper, this method is used to obtain  an $L^2$ version of M\"uller's type Atiyah-Patodi-Singer index theorem associated to the Dirac operator on a manifold with corner endowed with positive scalar curvature metrics on all boundary faces.

	The paper is organized as follows. In Section \ref{sec: preliminary}, we recall basic concepts, including geometric $C^*$-algebras and their smooth subalgebras, which will be used later in the paper. In Section \ref{sec:higher rho at infinity}, we define two secondary invariants at infinity for the Dirac operator on a complete manifold with uniform positive scalar curvature metric outside a compact set, the higher rho invariant at infinity and the delocalized eta invariant at infinity. We establish a formula for the delocalized eta invariant at infinity  in Subsection \ref{subsec:formulafordelocalizedeta}, together with which we generalize the Atiyah-Patodi-Singer index theorem. In Section \ref{sec manifold with corner} we apply the theory developed in Section \ref{sec:higher rho at infinity} to study invariants associated to the Dirac operator on a manifold with corner endowed with positive scalar curvature metrics on all boundary faces.

	\section{Preliminary}\label{sec: preliminary}
	
	In this section, we introduce some notions and concepts used in this paper, including geometric $C^*$-algebras and their smooth subalgebras. All the groups considered in this paper are finitely generated discrete groups. Denote by $|\cdot|$ the word length metric of a group $G$ constructed using some chosen finite set of generators. For an element $g$ in $G$, we say its conjugacy class, $\langle g \rangle$ has polynomial growth if there exist constants $C$ and $d$, such that
	\[
	\sharp \{ h\in \langle g \rangle, |h| \leq n   \}\leq Cn^d.
	\]

	\subsection{Geometric $C^*$-algebras}\label{subsec:geometric c algebra}
	
	In this subsection, we recall definitions of several geometric $C^*$-algebras, including equivariant Roe, localization and obstruction algebras (see  \cite{HigsonRoeYu, Roecoarsecohomoandindextheory, Yulocalizationalgeandbcconjec,   Roe2016positivecurvaturepartialvanishing, YuRufusbook} for more details).
	
	Let $X$ be a complete Riemannian manifold where $G$ acts properly, cocompactly and freely. An $X$-module is a separable Hilbert space equipped with a $*$ representation of $C_0(X)$. It is nondegenerate if the $*$ representation is nondegenerate. It is called standard if no nonzero function in $C_0(X)$ acts as a compact operator. Let $H_X$ be a standard nondegenerate $X$-module, where $H_X$ admits a unitary representation of the group $G$, and the representation of $C_0(X)$ is covariant to the group representation.
	
	We first recall some ingredients for constructing the geometric $C^*$-algebras.
	
	\begin{Def}
		let $H_X$ be an $X$-module and $T\in B(H_X) $ be a bounded linear operator.
		\begin{itemize}
			\item The propagation of $T$ is defined to be
			\[\sup{\{d(x,y)|\  (x,y)\in \text{Supp}(T)\}} ,\]
			where $\text{Supp}(T)$ is the complement in $X\times X$ of the set of points $(x,y)\in X\times X$ for which there exist $f,g\in C_0(X)$ satisfying $f(x)\neq 0,\ g(y)\neq 0$ such that $gTf=0$. $T$ is said to have  finite propagation if its propagation is finite. The propagation of a finite propagation operator $T$ is denoted as $\text{propagation}(T)$.
			\item $T$ is said to be locally compact if $fT$ and $Tf$ are compact for all $f\in C_0(X)$.
			\item $T$ is said to be pseudo-local if $[T,f]$ is compact for all $f\in C_0(X)$.
			\item Let $Z$ be a $G$-invariant subspace of $X$. We say $T$ is supported near $Z$ if there is $\lambda>0$, such that for all $f\in C_0(X)$ whose support is at least distance $\lambda$ away from $Z$, the operators $Tf$ and $fT$ are zero.
		\end{itemize}
	\end{Def}
	
	The following are definitions of the geometric $C^*$-algebras that will be used in this paper.
	
	\begin{Def}\label{def:C algebras}
		Let $H_X$ be a standard (also referred to as ample in the literature) nondegenerate $X$-module and $B(H_X)$ be the operator algebra of all bounded linear operators on $H_X$.
		\begin{itemize}
			\item $G$-equivariant Roe algebra $C^*(X)^G$ is the $C^*$-algebra generated by locally compact $G$-invariant operators with finite propagation.
			\item Let $Z$ be a $G$-invariant subspace of $X$. The localized $G$-equivariant Roe algebra at $Z$, $C^*(X, Z)^G$, is defined to be the $C^*$-algebra generated by $G$-invariant, locally compact operators with finite propagation, supported near $Z$.
			%\item $D^*(X)^G$ is the $C^*$-algebra generated by pseudo-local, finite propagation and $G$-invariant operators.
			\item $G$-equivariant localization algebra $C^*_L(X)^G$ is the $C^*$-algebra generated by all bounded, uniformly norm-continuous functions $f:[0,\infty)\to C^*(X)^G$ such that
			\[ \lim_{t\to \infty}\text{propagation of }f(t)=0. \]
			\item The kernel of the evaluation map
			\begin{eqnarray*}
				ev: C^*_L(X)^G &\to & C^*(X)^G\\
				f   & \to & f(0)
			\end{eqnarray*}
			is called the $G$-equivariant obstruction algebra, and denoted by  $C^*_{L,0}(X)^G$.
			% \item We say an operator $T\in B(H_X)$ is localizable if there is a sequence of pseudo-local, finite propagation operators $T_n$, such that
			%\[\lim_{n\to \infty} \|T-T_n\|=0.\]
			\item Let $Z$ be a $G$-invariant subspace of $X$, then $C^*_{L}(X,Z)^G$ (resp. $C^*_{L,0}(X,Z)^G$) is defined to be the closed subalgebra of $C^*_L(X)^G$ (resp. $C^*_{L,0}(X)^G$) generated by all elements $f$ such that there exists $c_t>0$ satisfying that $\lim_{t\to \infty} c_t=0$, and $\text{Supp} (f(t))\subset \{(x,z)\in X\times X|\  d((x,z),Z\times Z)\leq c_t\}.$
		\end{itemize}
	\end{Def}
	
	In general, if $Z$ is a cocompact $G$-invariant subspace of $X$, then we have the following isomorphism
	\[
	K_*(C^*(Z)^G)\cong K_*(C^*(X, Z)^G)\cong K_*(C^*_r(G)),
	\]
	where the first isomorphism is induced by the obvious embedding map.
	
	At the same time, by Lemma 3.10 of \cite{Yulocalizationalgeandbcconjec}, we know
	\begin{equation}\label{eq:iso of K of roe}
	K_*(C^*_{L}(X,Z)^G)\cong K_*(C^*_{L}(Z)^G).
	\end{equation}
	
	Moreover, by the following two $K$-theoretic six-term exact sequences
	\begin{eqnarray*}\label{eq:iso of K of loc}
	&&\xymatrix{
		K_0(C_{L,0}^*(Z)^G)\ar[r]& K_0(C_{L}^*(Z)^G)\ar[r] &K_0(C^*(Z)^G)\ar[d] \\
		K_1(C^*(Z)^G)\ar[u]&  K_1(C_{L}^*(Z)^G) \ar[l] & K_1(C_{L,0}^*(Z)^G)\ar[l]
	}\\
	&&
	\xymatrix{
		K_0(C_{L,0}^*(X,Z)^G)\ar[r]& K_0(C_{L}^*(X,Z)^G)\ar[r] &K_0(C^*(X,Z)^G)\ar[d] \\
		K_1(C^*(X,Z)^G)\ar[u]&  K_1(C_{L}^*(X,Z)^G) \ar[l] & K_1(C_{L,0}^*(X,Z)^G)\ar[l]
	}
	\end{eqnarray*}
	and a standard five lemma argument, one can see that
	\begin{equation}\label{eq:iso of K of obs}
	K_*(C^*_{L,0}(X,Z)^G)\cong K_*(C^*_{L,0}(Z)^G).
	\end{equation}
	Zeidler gave a constructive proof of the above isomorphisms in \eqref{eq:iso of K of roe}, and \eqref{eq:iso of K of obs}. See Lemma 3.7 of \cite{Zeidler} for details.

	\subsection{Engel-Samurka\c s smooth algebra}
	
	In this subsection, we introduce a $*$-algebra defined by Ka\v{g}an Samurka\c{s} in \cite{KaganSamboundsforfinitepart}, which is $KK$-equivalent to $C_r^*(G)$.  The construction of this algebra is inspired by Engel \cite{Engel2018}. It will be used in computing the paring of the delocalized trace and the higher index.

This algebra depends on the choice of a nontrivial element $g$ of $G$, whose conjugacy class has polynomial growth. For $S\subset G$
	and $f\in l^2 (G)$,  let $\|f\|_S$ be the $l^2$ norm of $f$ restricted to $S$. For $r>0$, set $B_r(S)$ as
	\[
	\{
	h\in G, \forall s\in S, d_G(e, hs^{-1})<r
	\},
	\]
	where $d_G$ is the word length metric.
	
	Choose $C>0$ and $d\in \mathbb{N}$ such that for all $k\in \mathbb{N}$,
	\[
	\sharp \{
	h\in \langle g \rangle, d_G(e,h) =k \}\leq Ck^d.
	\]
	For $a\in B(l^2(G))$ and $r>0$, define
	\[
	\mu_a(r)= \inf \{c>0, \forall f\in l^2(G), \|af\|_{G\backslash B_r(\text{supp} (f) )} \leq c \|f\|_G \},
	\]
	and
	\[
	\|a\|_g = \inf \{A>0, \forall r>0, \mu_a(r)\leq Ar^{-d/2+2}\}.
	\]
	Let $\|\cdot\|_{B(l^2(G))}$ be the operator norm on $B(l^2(G))$. For $f\in \mathbb{C}G$, we also denote  by $f\in B(l^2(G))$ the convolution operator by $f$ from the left.
	
	Define $C^{pol}_g(G)$ to be the completion of $\mathbb{C}G$ in the norm
	\[
	\|f\|= \|f\|_{B(l^2(G))}+ \|f\|_g.
	\]
	In \cite{KaganSamboundsforfinitepart}, Samurka\c s  showed that $C^{pol}_g(G)$ is closed under holomorphic functional calculus. We denote $C^{pol}_g(G)$ by  $\mathcal{A}_g(G)$.
	
	Let us recall the delocalized trace map $\tr_g$ from the Ka\v gan Samurka\c s algebra to $\mathbb{C}$.
	Denote by $\tr_g: \cA_g(G)\rightarrow\C$ the delocalized trace given by
	\begin{equation}\label{eq:orbitaltracemap}
		\tr_g\left(\sum_{\gamma\in G}a_{\gamma}\gamma\right)=\sum_{h\in \langle g \rangle} a_{h}
	\end{equation}
	where $\langle g\rangle $ stands for the conjugacy class of $g$ in $G.$ Moreover, the map $\tr_g$ induces the following delocalized trace map on $K$-theory:
	\[
	\tr_g : K_0(\mathcal{A}_g(G))\cong K_0(C_r^*(G))\to \mathbb{C}.
	\]

	At last, we recall the notion of $g$-trace class operators. Let $\mathcal{F}$  be a fundamental domain of $X$ with respect to the $G$-action. Compare the following definition to Definition 2.5 of \cite{HWW19equivariantAPSindexforproperaction}.
	
	\begin{Def}
		Let $T$ be a $G$-invariant operator having Schwartz kernel in $C^{\infty}(X\times X)$ and $g$ be an element of $G$. Then $T$ is $g$-trace class if
		\[
		\sum_{h\in \langle g \rangle} \int_{\mathcal{F}} |T(x, hx)| dx
		\]
		converges. The $g$-trace of $T$ is defined as
		\[
		\tr_g (T)\triangleq \sum_{h\in \langle g \rangle} \int_{\mathcal{F}} T(x, hx) dx.
		\]
	\end{Def}
	Note that when $G$ is trivial or $g=e$, $T$ is $g$-trace class if $T$ is trace class, which is equivalent to the condition $\int_{M}|T|(x,x)dx<\infty$. For a general $g$ distinct from identity, $\tr_g$ is not a positive trace, hence positivity is not required in the definition of $g$-trace class operators.
	
	\subsection{The Connes-Moscovici smooth subalgebra}
	
	In this subsection, let $M$ be a complete Riemannian manifold with a proper, free and cocompact action by a discrete group $G$. Let us recall Xie and Yu's construction of a particular smooth dense subalgebra of $C^*_{L,0}(M)^G$ (\cite{XYdelocalizedetainvalgebraicityandKtheoryofgroupCalgebras}).
	
	The following construction of a smooth dense subalgebra of $C^*_r(G)\otimes \mathcal{K}$ is due to Connes and Moscovici (cf. \cite{ConnesMoscycohomonovikovconj}).
	Let $\mathcal{R}$ be the algebra of smoothing operators on $M/G$. Under the isomorphism $L^2(M/G)\cong l^2(\mathbb{N})$, $\mathcal{R}$ is identified with the algebra of infinite matrices $(a_{ij})_{i,j\in \mathbb{N}}$, such that
	\[
	\forall k,l\in \mathbb{N},\ \ \ \sup_{i,j}i^kj^l |a_{ij}| < \infty.
	\]
	For a finitely generated discrete group $G$, let $\Delta_G $ be an unbounded operator on $l^2(G)$ defined by
	\[
	\Delta_G g= |g|g , \text{\ for\ }g\in G.
	\]
	Let $\Delta: l^2(\mathbb{N}) \to l^2(\mathbb{N})$ be the unbounded operator defined by
	\[
	\Delta (\delta_j) = j\delta_j.
	\]
	
	Denote by $\partial_G = [\Delta_G, \cdot] $ the unbounded derivation of $B(l^2(G))$, and $\widetilde{\partial}_G =  \partial_G \otimes I$ the unbounded derivation of $B(l^2 (G) \otimes l^2(\mathbb{N}))$. Set
	\[
	\mathcal{B}(M)^G= \{
	A\in C^*_r(G)\otimes \mathcal{K}, \text{ for all }k\in \mathbb{N}, \  \widetilde{\partial}^k_G(A) \circ (I \otimes \Delta )^2 \text{ is bounded }
	\},
	\]
	which can be proved to be a smooth dense subalgebra of $C^*_r(G)\otimes \mathcal{K}$. Lemma 2.7 of \cite{XYdelocalizedetainvalgebraicityandKtheoryofgroupCalgebras} shows that if $\langle g \rangle$ has polynomial growth, then the map $\tr_g$ can be extended to a trace map on $\mathcal{B}(M)^G$ such that for any $A=\sum_g A_g g \in \mathcal{B}(M)^G$, $\tr_g (A)$ equals
	\[
	\sum_{h\in \langle g \rangle} \text{trace}(A_h).
	\]
	Note that this $\tr_g$ map induces the same map on $K_0(C^*_r(G))$ with the corresponding one defined in the line \eqref{eq:orbitaltracemap}.
	
	In the meanwhile, the following algebra
	\[
	\mathcal{B}_{L, 0} (M)^G\triangleq \{
	f\in C^*_{L,0}(M)^G, \forall t\in [0,\infty), f(t)\in \mathcal{B}(M)^G
	\}
	\]
	is a smooth dense subalgebra of $C^*_{L,0}(M)^G$. Let $g\in G$ be a nontrivial element whose conjugacy class has polynomial growth. Define the map $\tau_g : \mathcal{B}_{L, 0} (M)^G\to \mathbb{C}$ to be
	\[
	\tau_g (A) = \frac{1}{2\pi i}\int_{0}^\infty \tr_g (\frac{d A(t)}{d t} A^{-1}(t)) dt.
	\]
	As shown in \cite{XYdelocalizedetainvalgebraicityandKtheoryofgroupCalgebras}, $\tau_g$ induces the following determinant map
	\begin{equation}\label{eq:determinantmap}
		\tau_g: K_1 (\mathcal{B}_{L, 0} (M)^G)\cong K_1(C^*_{L,0} (M)^G) \rightarrow \mathbb{C},
	\end{equation}
	which is also denoted as $\tau_g$.
	
	One may think it is sufficient to recall the definition of either the Engel-Samukar\c s smooth subalgebra or the  Connes-Moscovici smooth subalgebra, however, there is a reason for us not to do so.
	The advantage of the Engel-Samukar\c s smooth subalgebra is that it allows us to do analysis in a noncocompact setting and to apply the method of Hochs, Wang and Wang (cf. \cite{HWW19equivariantAPSindexforproperaction}) to develop the index formula for the Dirac operator on noncocompact complete manifold with uniform positive scalar curvature metric at infinity, while the Connes-Moscovici smooth subalgebra is necessary for us to apply the theory of Xie and Yu on delocalized trace (cf. \cite{XYdelocalizedetainvalgebraicityandKtheoryofgroupCalgebras}) to study the higher rho invariant at infinity and define the delocalized eta invariant at infinity, where a cocompact setting is sufficient.

	\section{Secondary invariants at infinity}\label{sec:higher rho at infinity}
	In this section, we introduce two new secondary invariants for the Dirac operator on a complete spin manifold endowed with a metric with uniform positive scalar curvature outside a cocompact set, i.e. the higher rho invariant at infinity and the delocalized eta invariant at infinity. We begin with two approaches to the definition of the higher index of the Dirac operator. We also establish a formula for the delocalized eta invariant, along with which we generalize the Atiyah-Patodi-Singer index theorem.
	
	\subsection{Higher index}\label{sec:higher index a definition}
	
	In this subsection, we define the higher index of the Dirac operator on a complete manifold having a metric with uniform positive scalar curvature outside a cocompact set. Higher index of the Dirac operator on a complete manifold with positive scalar curvature at infinity was first introduced by Bunke in \cite{Bunke95}. There is a nice description of this higher index in \cite{BlockWeinberger}. However, our construction in this subsection follows  \cite{Roecoarsecohomoandindextheory} \cite{Roe2016positivecurvaturepartialvanishing}.
	
	Let $X$ be a complete Riemannian manifold (even dimensional) with a spin structure on which a discrete group $G$ acts freely, properly and preserving the metric.
	Let $M\subset X$ be a $G$-invariant subset.
	Assume that there is a metric $\mu$ with uniform positive  scalar curvature  with strictly positive lower bound $h_0>0$. on $X\backslash M$.
	
	\begin{figure}[h]
		\centering
		\begin{tikzpicture}[scale=0.5]
		\draw[thick](18,6)to[out=225,in=90](6,0)to[out=270,in=150](18,0);
		\draw[thick](18,4)to[out=225,in=90](14,2)to[out=270,in=150](18,2);
		\draw[thick](9,1.5)to[out=210,in=150](9,0.5);
		\draw[thick](8.7,1.7)to[out=350,in=120](9,1.5)to[out=300,in=40](9,0.5)to[out=220,in=30](8.7,0.3);
		\draw[dashed,thick](10,3.5)to[out=300,in=60](10,-1);
		\node at (8.5,0) {$M$};
		\node at (13,2) {$X\backslash M$};
		%\node at (0,0) {$X$};
		%\draw [->](4,0)--(5.5,0);
		\end{tikzpicture}
		\caption{Complete manifold $X$, with $G$ cocompact subset $M$.}
		\label{fig:X with compact M}
	\end{figure}
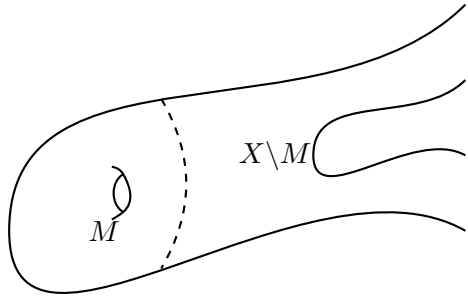

	Let $D$ be the spin Dirac operator associated to the metric $\mu$ on $X$. Denote by $D_M, D_c$ the restriction of $D$ to $M$ and $C:=X\backslash M$ respectively.
	Because $X$ has even dimension, the spinor bundle $S$ is $\Z_2$-graded, i.e. $S=S^+ \oplus S^-$, and the Dirac operator is odd with respect to the grading $L^2(X, S^+)\oplus L^2(X, S^-)$:
	\[
	D=\begin{bmatrix}0 & D^- \\ D^+ & 0\end{bmatrix} \qquad (D^+)^*=D^-.
	\]
	
	Let $b: (-\infty, +\infty)\to [-1,1]$ be an odd smooth function such that $b(x)=1$ when $x> \frac{\sqrt {h_0}}{2}$. Lemma 2.1 and Lemma 2.3 of \cite{Roe2016positivecurvaturepartialvanishing} state that $b(D)$ is pseudo-local,  and $b(D)^2-1$ lies in $C^*(X,M)^G$. Since $b$ is an odd function, $b(D)$ can also be decomposed as
	\[
	\begin{bmatrix}
	0 & b_+ \\
	b_- & 0
	\end{bmatrix}.
	\]
	Let $V$ be
	\[
	\begin{bmatrix}
	1 & b_+ \\
	0 & 1
	\end{bmatrix}\begin{bmatrix}
	1 & 0 \\
	-b_- & 1
	\end{bmatrix}\begin{bmatrix}
	1 & b_+ \\
	0 & 1
	\end{bmatrix}\begin{bmatrix}
	0 & 1 \\
	-1 & 0
	\end{bmatrix}.
	\]
	
	\begin{Def}\label{def: index I}
		The $G$-equivariant coarse index of $D$, $\ind_{G,X,M}(D)$ is defined to be the $K$-theory class in $K_0(C^*(X,M)^G)$, represented by the formal difference of idempotents
		\[
		V\begin{bmatrix}
		1 & 0\\
		0 & 0
		\end{bmatrix}V^{-1} - \begin{bmatrix}
		1 & 0\\
		0 & 0
		\end{bmatrix}.
		\]
		Furthermore, if $M/G$ is compact, the $K$-theory class \[\ind_{G,X,M}(D)\in K_0(C^*(X,M)^G)\cong K_0(C^*_r(G)),\]
		is called the higher index of $D$, and denoted as $\ind_{G}(D)$.
	\end{Def}
	
	Note that direct computation shows that
	\[
	V\begin{bmatrix}
	1 & 0\\
	0 & 0
	\end{bmatrix}V^{-1}=
	\begin{bmatrix}
	(1-b_+b_-)b_+b_-+b_+b_- & (2-b_+b_-)b_+(1-b_-b_+)\\ b_-(1-b_+b_-)& (1-b_-b_+)^2
	\end{bmatrix}.
	\]

	\subsection{Higher index: a different approach}\label{subsec: higher index different approach}
	
	In this subsection, we introduce a new approach to define the higher index of the Dirac operator on a complete manifold with metric admitting uniform positive scalar curvature at infinity. This approach is necessary for us to define the higher rho invariant at infinity. This approach is inspired by \cite{XieandYu14positivescalarcurvature}, \cite{Roe2016positivecurvaturepartialvanishing}, and has been used to define the higher index in \cite{ZhangLiuLiu}. See Willett and Yu's recently published book \cite{YuRufusbook} for more details.

	Set $X_{[r,s]},\ r,s\in [0, \infty],$ to be the $G$-invariant subset of $X$ defined by
	\[
	\{x\in X | r\leq \text{dist}(x, M)\leq s\}.
	\]
	For simplicity, $X_{[0, s]}, \  s< \infty,$ is denoted as $X_{\leq s}$, and $X_{[r, \infty]}, \ r>0$, is denoted as $X_{\geq r}$. Note that $X_{[0, \infty]}= X$. The set $X_{(r,s)} $ are defined similarly.
	
	Let $\chi: (-\infty, \infty)\to [-1, 1]$ be a normalizing function satisfying the following conditions,
	\begin{enumerate}
		\item $\chi$ is a smooth odd function, such that $\lim\limits_{s\to \pm \infty} \chi(s)=\pm 1$.
		\item $\lim\limits_{s\to 0} \frac{\chi(s)}{s}=1$.
		\item $\hat \chi$ is a compactly supported  distribution on $\mathbb{R}$, such that its support contains $0$, and the diameter of its support is bounded by  a real number $\delta$, i.e.  $\text{diam}(\text{Supp}(\hat \chi))\leq \delta$.
		\item Let $b: (-\infty, \infty)\to [-1, 1]$
be the normalizing function such that  $b(x)=1$ when $x> \frac{\sqrt {h_0}}{2}$. Then $\|\chi-b\| \leq \epsilon$, where $\epsilon$ is a positive number less than  $\frac{1}{10000^{10000}}$.
	\end{enumerate}
	Furthermore, let $T\geq 1$ be a real number.
	
	Define $F_T$ to be the self-adjoint bounded operator
	\[\chi(TD)=\int_{-\infty}^\infty \hat\chi(s) e^{2\pi i sTD} ds. \]
	
	Since $\chi$ is an odd function, $F_T$ is odd with respect to the grading $L^2(X, S^+)\oplus L^2(X, S^-)$:
	\[
	F_T=\begin{bmatrix}
	0 & U_T \\
	U^*_T & 0
	\end{bmatrix}.
	\]
	
	Let $N$ be a sufficiently large integer such that for any $x\in [-1000, 1000]$, there are
	\[
	|\sum_{n=1}^N   \frac{(2\pi i)^n}{n!}| \leq  \frac{\epsilon }{1000000}, \ \ \ |
	\sum_{n=N}^\infty \frac{(2\pi i x)^n}{n!}
	|\leq  \frac{\epsilon }{1000000}.
	\]
	
	Let $r>0$ be a positive number such that
	\begin{equation}%\label{eq:a choice of r}
	Nr \leq \frac{\delta}{100}.
	\end{equation}
	For any $n\in \mathbb{N}$, we choose a $G$-invariant locally finite open cover $\{U_{n,j}\}_j$, and a $G$-invariant partition of unity $\{\phi_{n,j}\}_j$ subordinate to  $\{U_{n,j}\}_j$, such that
	\begin{enumerate}
		\item if $U_{n,j}\subset X_{\leq 100T\delta}$, then the diameter of $U_{n,j}$ is less than $\frac{r}{2^n}$.
		\item for any fixed $n$, there are precisely two open sets $U_{n,j}$ such that $U_{n,j}\bigcap X_{\geq 100T\delta} $ is nonempty. For convenience, we denote them by $W_{n,1}= X_{(100T\delta-\frac{r}{2^n}, 110T\delta)}$ and $W_2= X_{> 100T\delta}$. %where $C_T$ is a sufficiently large positive number.
	\end{enumerate}
	
	Set
	\[
	F_{T,n}= \sum_{j} \sqrt{\phi_{n,j}} F_T \sqrt{\phi_{n,j}}, \ \forall n\in \mathbb{N},
	\]
	and
	\[
	F_T(t)= (n+1-t)F_{T,n}  + (t-n)F_{T,n+1}.
	\]
	For any $t$, $F_T(t)$ is odd with respect to the grading $L^2(X, S^+)\oplus L^2(X, S^-)$:
	\[
	F_T(t)=\begin{bmatrix}
	0 & U_T(t) \\
	U^*_T(t) & 0
	\end{bmatrix}.
	\]
	
	As shown by \cite[Lemma 2.6]{Yunovikovfiniteasymptoticgroup}, $\|U_T(t)\|\leq 4\|U_T\|=4$.
	Consider
	\[
	W_T(t)= \begin{bmatrix}
	1 & U_T(t) \\
	0 & 1
	\end{bmatrix}\begin{bmatrix}
	1 &  0\\
	-U_T^*(t) & 1
	\end{bmatrix}\begin{bmatrix}
	1 & U_T(t) \\
	0 & 1
	\end{bmatrix}\begin{bmatrix}
	0 & 1 \\
	-1 & 0
	\end{bmatrix}.
	\]

	Note that $F_T$ has finite propagation $T\delta$.
	
	Decompose $L^2(X, S)$ into $L^2(X_{\leq 300T\delta},S)\oplus L^2(X_{\geq 300T\delta},S) $.
	
	\begin{Lem}\label{lem:cut of coarse index}
		Write
		\[
		U_T(t)U_T^*(t)=\begin{bmatrix}
		B_{11}(t) & B_{12}(t)\\
		B_{21}(t) & B_{22}(t)
		\end{bmatrix}\ \text{and} \ U_T^*(t)U_T(t)=\begin{bmatrix}
		C_{11}(t) & C_{12}(t)\\
		C_{21}(t) & C_{22}(t)
		\end{bmatrix}
		\]
		with respect to the decomposition
		\[L^2(X, S)=L^2(X_{\leq 300T\delta},S)\oplus L^2(X_{\geq 300T\delta},S). \]
		Then one has the following inequalities
		\begin{eqnarray*}
			&\|B_{12}(t)\|=\|B_{21}(t)\|\leq \epsilon,\  \|B_{22}(t)-1\|\leq \epsilon, \\
			&\|C_{12}(t)\|=\|C_{21}(t)\|\leq \epsilon,\  \|C_{22}(t)-1\|\leq \epsilon.
		\end{eqnarray*}
		for any $t\in [0, \infty)$.
	\end{Lem}
	\bpr
	By assumption, $D_c^2$ is bounded below by $\frac{h_0}{4}$. Therefore it has a Friedrich's extension $E$ on the Hilbert space $L^2(X\slash M,S)$, which is a selfadjoint operator bounded below by $\frac{h_0}{4}$. For the normalizing function $\chi$, define an operator $\chi^2(T\sqrt{E})$.
	
	Since the propagation of $F_T$ is controlled by $T\delta$, for all $f\in L^2(X_{\geq 120T\delta},S) $, we have
	\[
	F_T(t)(f)\in L^2(X_{\geq 119T\delta},S),\ \ F_T(t)^2(f)\in L^2(X_{\geq 118T\delta},S).
	\]
	By a standard finite propagation argument (cf. the proof of \cite[Lemma 2.5]{Roe2016positivecurvaturepartialvanishing}), we have
	\[
	\|F_T(t)^2(f)-f\| = \| \chi^2(T\sqrt{E}) (f)-f\| \leq  \epsilon \|f\|,
	\]
    where the last inequality follows from condition $4)$ in the definition of $\chi$, which says that $\chi^2$ is closed to $b^2$, a function equals $1$ on the spectrum of $T\sqrt{E}$.
	This implies the Lemma.
	\epr
	
	%For any $t\in [0, \infty)$, denote by $W_T(t)$ the following invertible operator
	%\[
	%W_T(t)= \begin{bmatrix}
	%1 & U_T(t) \\
	%0 & 1
	%\end{bmatrix}\begin{bmatrix}
	%1 &  0\\
	%-U_T^*(t) & 1
	%\end{bmatrix}\begin{bmatrix}
	%1 & U_T(t) \\
	%0 & 1
	%\end{bmatrix}\begin{bmatrix}
	%0 & -1 \\
	%1 & 0
	%\end{bmatrix}.
	%\]
	Now let us consider the formal difference
	\[
	W_T(t)\begin{bmatrix}
	1 &0 \\ 0 & 0
	\end{bmatrix} W_T^{-1}(t)-\begin{bmatrix}
	1 &0 \\ 0 & 0
	\end{bmatrix}.
	\]
	By definition of the boundary map of the $K$-theory six-term exact sequence, this formal difference represents a $K$-theory class in $K_0(C^*(X)^G)$, which is not isomorphic
	to $K_0(C^*_r(G))$. To define the equivariant coarse index of $D$, we construct an almost idempotent
	sufficiently close to
	\[
	W_T(t)\begin{bmatrix}
	1 &0 \\ 0 & 0
	\end{bmatrix} W_T(t)^{-1}
	\]
	in the operator norm.

Note that
\[
W_T^{-1}(t)=
	\begin{bmatrix}
	0 & -1 \\
	1 & 0
	\end{bmatrix}\begin{bmatrix}
	1 & -U_T(t) \\
	0 & 1
	\end{bmatrix}\begin{bmatrix}
	1 &  0\\
	U_T^*(t) & 1
	\end{bmatrix}\begin{bmatrix}1 & -U_T(t) \\
	0 & 1
	\end{bmatrix}.
\]
	Direct computation shows that
	\begin{eqnarray*}
		&& W_T(t)\begin{bmatrix}
			1 &0 \\ 0 & 0
		\end{bmatrix} W_T(t)^{-1}\\
		&=&\begin{bmatrix}
			(1-U_T(t)U_T^*(t))U_T(t)U^*_T(t)+U_T(t)U^*_T(t) & (2-U_T(t)U^*_T(t))U_T(t)(1-U^*_T(t)U_T(t))\\ U_T(t)^*(1-U_T(t)U^*_T(t))& (1-U^*_T(t)U_T(t))^2
		\end{bmatrix}\\
		&=& \begin{bmatrix}
			1-(1-U_T(t)U^*_T(t))^2 & (2-U_T(t)U^*_T(t))U_T(t)(1-U^*_T(t)U_T(t))\\ U^*_T(t)(1-U_T(t)U^*_T(t))& (1-U^*_T(t)U_T(t))^2
		\end{bmatrix}.
	\end{eqnarray*}
	
	Set
	\[
	Z_1(t)= \begin{bmatrix}
	1-B_{11}(t) &0\\
	0 & 0
	\end{bmatrix},\ \text{and} \ Z_2(t)= \begin{bmatrix}
	1-C_{11}(t) &0\\
	0 & 0
	\end{bmatrix}.
	\]
	By Lemma \ref{lem:cut of coarse index}, we have
	\[
	\|Z_1(t) -(1-U_T(t)U^*_T(t))\|\leq 3 \epsilon, \text{and} \ \|Z_2(t) -(1-U^*_T(t)U_T(t))\|\leq 3\epsilon.
	\]
	Define $P_T(t)$ to be the operator
	\begin{equation}\label{eq:PTt the index}
		\begin{bmatrix}
			1-Z_1^2(t) & (2-U_T(t)U^*_T(t))U_T(t) Z_2(t)\\
			U^*_T(t) Z_1(t) & Z_2^2(t)
		\end{bmatrix}.
	\end{equation}
	Then we have
	\[
	\|P_T(t)-W_T(t)\begin{bmatrix}
	1 &0 \\ 0 & 0
	\end{bmatrix} W_T(t)^{-1}\|\leq 2000\epsilon,
	\]
	which implies
\begin{eqnarray*}	
&&\|P_T(t)-P_T(t)^2\|\\
&\leq &\|P_T(t)-W_T(t)\begin{bmatrix}
	1 &0 \\ 0 & 0
	\end{bmatrix} W_T(t)^{-1}\|+\|(W_T(t)\begin{bmatrix}
	1 &0 \\ 0 & 0
	\end{bmatrix} W_T(t)^{-1})^2-P_T(t)^2\|\\
&\leq &\|P_T(t)-W_T(t)\begin{bmatrix}
	1 &0 \\ 0 & 0
	\end{bmatrix} W_T(t)^{-1}\|+\|W_T(t)\begin{bmatrix}
	1 &0 \\ 0 & 0
	\end{bmatrix} W_T(t)^{-1}-P_T(t)\|\|P_T(t)\|\\
&&+\|W_T(t)\begin{bmatrix}
	1 &0 \\ 0 & 0
	\end{bmatrix} W_T(t)^{-1}\|\|W_T(t)\begin{bmatrix}
	1 &0 \\ 0 & 0
	\end{bmatrix} W_T(t)^{-1}-P_T(t)\|\\
&\leq & 5\times 2000\epsilon=10000\epsilon.
\end{eqnarray*}
	The following lemma explains why we construct $P_T(t)$ in Equation \eqref{eq:PTt the index}.
	
	\begin{Prop}\label{prop:cut of coarse index}
		Let $P_T(t)$ be as in Equation \eqref{eq:PTt the index}, then $P_T(t)$ preserves the decomposition
		\[
		L^2(X, S)=L^2(X_{\leq 300T\delta}, S)\oplus L^2(X_{\geq 300T\delta}, S).
		\]
		Moreover, $P_T(t)$ has the form
		\[
		P_T(t)= \begin{bmatrix}
		P_T'(t) & 0\\
		0 &  \begin{bmatrix}
		1 & 0 \\ 0 & 0
		\end{bmatrix}
		\end{bmatrix}
		\]
		with respect to the decomposition
		\[
		L^2(X, S)=L^2(X_{\leq 300T\delta},S)\oplus L^2(X_{\geq 300T\delta}, S).
		\]
		
	\end{Prop}
	
	\bpr
	The lemma is proved by direct computation.
	
	For any $f\in L^2(X_{\geq 300T\delta}, S)$, by definitions of $Z_1$ and $Z_2$, we have
	\begin{eqnarray*}
		&& Z_1^2(t) f = 0,\\
		&&(2-U_T(t)U^*_T(t))U_T(t) Z_2(t) f=0,\\
		&& U^*_T(t) Z_1(t) f=0,\\
		&& Z_2^2(t)f=0.
	\end{eqnarray*}
	On the other hand, for any $f\in L^2(X_{\leq 300T\delta}, S)$, we have
	\begin{eqnarray*}
		&& Z_1(t) f \in   L^2(X_{\leq 300T\delta}, S),\\
		&& Z_1^2(t) f \in   L^2(X_{\leq 300T\delta}, S),\\
		&& Z_2(t) f \in   L^2(X_{\leq 300T\delta}, S),\\
		&& Z_2^2(t)f \in   L^2(X_{\leq 300T\delta}, S).
	\end{eqnarray*}
	For any $t$, by definition of $U_T(t)$, we know that the propagation of $U_T(t)$ and $U^*_T(t)$ is less than $T\delta$. Thus one can see
	\begin{eqnarray*}
		&& (2-U_T(t)U^*_T(t))U_T(t) Z_2(t) f \in   L^2(X_{\leq 300T\delta}, S),\\
		&& U^*_T(t)Z_1(t)f \in   L^2(X_{\leq 300T\delta}, S).
	\end{eqnarray*}
	This completes the proof.
	\epr
	
	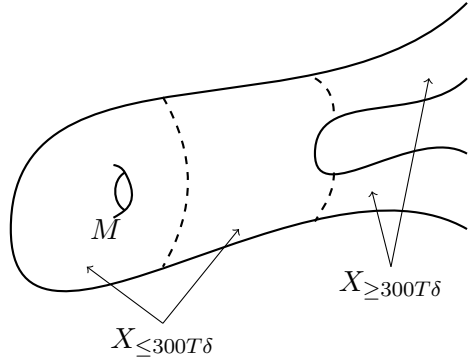
\begin{figure}[h]
		\centering
		\begin{tikzpicture}[scale=0.5]
		\draw[thick](18,6)to[out=225,in=90](6,0)to[out=270,in=150](18,0);
		\draw[thick](18,4)to[out=225,in=90](14,2)to[out=270,in=150](18,2);
		\draw[thick](9,1.5)to[out=210,in=150](9,0.5);
		\draw[thick](8.7,1.7)to[out=350,in=120](9,1.5)to[out=300,in=40](9,0.5)to[out=220,in=30](8.7,0.3);
		\draw[dashed,thick](10,3.5)to[out=300,in=60](10,-1);
		\draw[dashed,thick](14,4)to[out=330,in=450](14.5,3);
		\draw[dashed,thick](14.5,1.5)to[out=270,in=60](14,0.2);
		\node at (8.5,0) {$M$};
		\node at (10,-3) {$X_{\leq 300T\delta}$};
		\draw[->](10,-2.5)--(8,-1);
		\draw[->](10,-2.5)--(12,0);
		\node at (16,-1.5) {$X_{\geq 300T\delta}$};
		\draw[->](16,-1)--(15.5,1);
		\draw[->](16,-1)--(17,4);
		%\node at (0,0) {$X$};
		%\draw [->](4,0)--(5.5,0);
		\end{tikzpicture}
		\caption{For all $t\in[0, \infty)$, the operator $P_{X,M,T}(t)$ is supported on $X_{\leq 300T\delta}$.}
		\label{fig:support relind representative}
	\end{figure}
	
Since  $\|P_T'(t)^2-P_T'(t)\|\leq 10000\epsilon$, the spectrum of $P_T'(t)$ lies in
    \[
    O(0,20000\epsilon)\cup O(1, 20000\epsilon).
    \]
Note that $20000\epsilon$ is far less than $\frac{1}{4}$, which implies that
\[
O(0,20000\epsilon)\cap O(1, 20000\epsilon)=\emptyset.
\]
Thus
\[
		P_T'(t)-\begin{bmatrix}
		1 & 0 \\
		0 & 0
		\end{bmatrix}
		\]
defines a class in $K_0(C^*(X_{\leq 300T\delta})^G)$. In fact, set $\mathcal{H}\equiv 0$ on $O(0,20000\epsilon)$ and $\mathcal{H}\equiv 1$ on $O(1,20000\epsilon)$, then
\[
\mathcal{H}(P_T')(t)=\frac{1}{2\pi i}\int_\mathcal{C} \frac{\mathcal{H}(z)}{z-P_T'(t)}dz,
\]
where $\mathcal{C}$ is a contour surrounding the spectrum of $P_T'(t)$ in $O(0,20000\epsilon)\cup O(1, 20000\epsilon)$, defines a genuine idempotent.
	\begin{Def}\label{def: index II}
		By Proposition \ref{prop:cut of coarse index}, for any $t\in [0, \infty)$,
		\[
		P_T'(t)-\begin{bmatrix}
		1 & 0 \\
		0 & 0
		\end{bmatrix}
		\]
		represents an element in $K_0(C^*(X_{\leq 300T\delta})^G)$, which is denoted by $\ind_{X,M,T}(D)$.  (see Figure \ref{fig:support relind representative}).
		
		When $M/G$ is compact, $\ind_{X,M ,T}(D)$ defines an element in $K_0(C^*_r(G))$ and coincides with the higher index defined in Definition \ref{def: index I}.
		%\[
		%\ind_{X,M,T}(D)= \ind_{G}(D)\in K_0(C^*_r(G)).
		%\]
	\end{Def}
	%\begin{Def}\label{def:relative index}
	%	Let $P_T'$ be the operator defined in Proposition \ref{prop:cut of coarse index}. By definition,
	%	\[
	%	P_T'-\begin{bmatrix}
	%	1 & 0 \\
	%	0 & 1
	%	\end{bmatrix}
	%	\]
	%	represents an element in $K_0(C^*(X_{\leq 300T\delta})^G).$
	%	We call this element the relative equivariant coarse index, or simply relative coarse index of $D$.
	%	In the following, we denote the relative coarse index of $D$ by $\ind_c (D,T)$.
	%	For a fix $T$, $\ind_c (D,T)$ is independent from the choice of normalizing function $\chi.$
	%	Furthermore, under the isomorphism
	%	\[
	%	K_0(C^*(X_{\leq 300T\delta})^G)\cong K_0(C^*_r(G)),
	%	\]
	%	$\ind_c (D,T)$ gives us an element in $ K_0(C^*_r(G))$, which we call
	%	the relative higher index of $D$, or higher index of $D$ for short, and denote as $\ind_G D$. Note that $\ind_G D $ does not depend on the choice of $T$ and $\chi$.
	%\end{Def}
	
In addition, it is straightforward to see that
\[
\|P_T'(t)-\mathcal{H}(P_T')(t)\|\leq 20000\epsilon.
\]

Then we have the following estimate which is useful in the next subsection:
\begin{eqnarray*}
&&\|e^{2\pi i P_T'(t)}-e^{2\pi i \mathcal{H}(P_T')(t)}\|\\
&=&\| \frac{1}{2\pi i}\int_{\|z\|=4} \frac{e^{2\pi i z}}{z-P_T'(t)}dz- \frac{1}{2\pi i}\int_{\|z\|=2} \frac{e^{2\pi i z}}{z-\mathcal{H}(P_T')(t)}dz\|\\
&\leq & 20000\epsilon\times  \frac{4}{2\pi}\int_{\|z\|=4}1dz\\
&=& 2^5\times 10^5 \epsilon  \leq  500000\epsilon.
\end{eqnarray*}

	\subsection{Higher rho invariant at infinity}\label{subsec higher rho at infinity}
	
	In this subsection, we define the higher rho invariant at infinity. This invariant gives a $K$-theoretic  demonstration of the nonlocality of the higher index of the Dirac operator on noncompact spin manifold endowed with a metric admitting uniform positive scalar curvature at infinity.
	
	The following theorem is necessary in order to introduce the definition of the higher rho invariant at infinity.
	
	\begin{Thm}\label{thm:support of higher rho at infinity}
		Let $\ind_{X,M,T}(D)$ be as in Definition \ref{def: index II}, then
		$
		\partial (\ind_{X,M,T}(D))$ represents an element in $ K_1(C^*_{L,0}(X_{[90T\delta, 300T\delta]})^G),
		$ 	
		where $\partial $ is the connecting map in the following exact sequence
		\[
		\xymatrix{
			K_0(C^*_{L,0}(X_{\leq  300T\delta})^G) \ar[r] &  K_0(C^*_{L}(X_{\leq 300T\delta})^G) \ar[r] & K_0(C^*(X_{\leq 300T\delta})^G)\ar[d]\\
			K_1(C^*(X_{\leq 300T\delta})^G)\ar[u] & K_1(C^*_{L}(X_{\leq 300T\delta})^G)\ar[l] & K_1(C^*_{L,0}(X_{\leq 300T\delta})^G)\ar[l]
		}.
		\]
	\end{Thm}
	
	\bpr
	Fix $t=0$.
	To compute $\partial (\ind_{X,M,T}(D))$, one needs to lift $P_{X,M,T}(0)$ to $C^*_{L}(X_{\leq 300T\delta})^G$ first.
	
	Recall that $N$ is an integer sufficiently large such that for any $x\in [-1000, 1000]$,
	\[
	|\sum_{n=1}^N   \frac{(2\pi i)^n}{n!}| \leq  \frac{\epsilon }{1000000}, \ \ \ |
	\sum_{n=N}^\infty \frac{(2\pi i x)^n}{n!}
	|\leq  \frac{\epsilon }{1000000}.
	\]
	and $r$ is a positive number such that
		\begin{equation}\label{eq:a choice of r}
	Nr \leq \frac{\delta}{100}.
	\end{equation}
	Thus for any  $x\in [-1000, 1000]$, we have
    \[
    |1 + \sum_{n=1}^N \frac{(2\pi i x)^n}{n!}
	 -e^{2\pi i x}|= |
	\sum_{n=N}^\infty \frac{(2\pi i x)^n}{n!}
	|< \frac{\epsilon}{2}.
    \]
    Furthermore, there is
    \[
   | 1 + \sum_{n=1}^N \frac{(2\pi i x)^n}{n!}-e^{2\pi i x}-\sum_{n=1}^N   \frac{(2\pi i)^n}{n!}x| \leq \frac{\epsilon}{2}+\frac{\epsilon}{2}=\epsilon.
    \]
    However, we have
    \[
     \sum_{n=1}^N \frac{(2\pi i x)^n}{n!}-\sum_{n=1}^N  \frac{(2\pi i)^n}{n!}x=\sum_{n=2}^N\left( \sum_{j=0}^{n-2} \frac{(2\pi i )^n}{n!} x^j \right)(x^2-x).
    \]
    These above equalities imply that
	\[
	|
	1 + \sum_{n=2}^N\left( \sum_{j=0}^{n-2} \frac{(2\pi i )^n}{n!} x^j \right) \left(x^2-x\right)-e^{2\pi i x}|\leq \epsilon.
	\]

	For each $t\in [0, \infty)$, consider a $G$-invariant partition of unity $\{\zeta_t, \theta_t\}$ on $X_{\leq 300T\delta}$ satisfying
	\begin{enumerate}
		\item $\zeta_t+\theta_t =1$.
		\item for $n\leq t <n+1$, $\zeta_t \equiv 1 $ on $X_{\leq 100T\delta- r/2^{n-1} }$ and $\zeta_t \equiv 0$ on $X_{\geq 100T\delta- r/2^n  }$.
	\end{enumerate}
	
	Let $\xi\in C^\infty (-\infty, \infty) $ be a decreasing function such that $\xi|_{[-\infty, 1]}\equiv 1$ and $\xi|_{[2, \infty]}\equiv 0$.
	
	By definition, one can verify that the path $t\mapsto P'_{X,M,T}(t)$ forms a lifting of $P_{X,M,T}(0)$ in $C^*_L(X_{\leq 300T\delta})^G$, where $P'_{X,M,T}(t)$ is defined as
	\[
 \left\{
	\begin{array} {cc}
	(1-t) P_{X,M,T}(0) + t(\sqrt{\zeta_0} P_{X,M,T}(0) \sqrt{\zeta_0} + \sqrt{\theta_0} P_{X,M,T}(0) \sqrt{\theta_0}  ), & t\in [0, 1],\\
	\begin{matrix}
	\begin{bmatrix}
	1 & 0\\
	0 & 0
	\end{bmatrix} + \sqrt{\zeta_{t-1}} (P_{X,M,T}(t-1)-\begin{bmatrix}
	1 & 0\\
	0 & 0
	\end{bmatrix}   ) \sqrt{\zeta_{t-1}}\\
	+  \xi(t) \sqrt{\theta_{t-1}}( P_{X,M,T}(t-1)-\begin{bmatrix}
	1 & 0\\
	0 & 0
	\end{bmatrix}) \sqrt{\theta_{t-1}}
	\end{matrix},
	& t\in [1, \infty).
	\end{array}
	\right.
	\]
	
	Then we have
	\[
	\partial (\ind_{X,M,T}(D)) = [e^{2\pi i P'_{X,M,T}(\cdot)}].
	\]
	
	Define a path of invertible operators $u_T(t)$ as
	\[
	1 + \sum_{n=2}^N\left( \sum_{j=0}^{n-2} \frac{(2\pi i )^n}{n!} (P'_{X,M,T}(t))^j \right) \left(\left(P'_{X,M,T}(t)\right)^2-\left(P'_{X,M,T}(t)\right)\right).
	\]
	It is straight-forward to verify that
	\[
	\|u_T(t)-e^{2\pi iP'_{X,M,T}(t)}\|\leq \epsilon.
	\]
	Thus $u_T(t)$ is invertible for any $t$ and the path
	\[
	u'_T(t)= \left\{
	\begin{array}{cc}
	(1-t) +tu_T(0)   & t\in[0,1] \\
	u_T(t-1) & t\in [1,\infty)
	\end{array}
	\right.
	\]
	represents the same element in $K_1(C^*_{L,0}(X)^G )$ as the path $e^{2\pi i P'_{X,M,T}(\cdot)}$. In other words, we have
	\[
	\partial (\ind_{X,M,T}(D)) = [u'_T(t)].
	\]
	
	Note that the propagation of $u'_T(\cdot)$ restricted to $X_{\leq 98 T\delta}$ is less than $\frac{\delta}{100}$ by the choice of $r$ in line \eqref{eq:a choice of r}.
	For any $h\in L^2 (X_{\leq 80 T\delta}, S)$ and any $t\in [0,\infty)$, we have both $(u'_T(t) )^2h\in  L^2 (X_{\leq 98 T\delta}, S)$. Recall that
\[\mathcal{H}(P_T')(t)=\frac{1}{2\pi i}\int_\mathcal{C} \frac{\mathcal{H}(z)}{z-P_T'(t)}dz\]
is a path of genuine idempotent obtained from $P_T'(t)$ as above.
Thus $\forall h\in L^2 (X_{\leq 80 T\delta}, S)$, we have
 %As mentioned in Definition \ref{def: index II}, $P_{X,M,T}(t)$ is a genuine projection. Note that $P'_{X,M,T}(t)=P_{X,M,T}(t)$ restricting to $L^2 (X_{\leq 98 T\delta}, S)$. Thus $e^{2\pi i P'_{X,M,T}(t)}$ equals identity restriction to $L^2 (X_{\leq 80 T\delta}, S)$. One can see that for any $h\in L^2 (X_{\leq 80 T\delta}, S)$
	%\begin{eqnarray*}
%		&&\|u'_T(t)h - h\| \\
%		&=& \| h + \sum_{n=2}^N\left( \sum_{j=0}^{n-2} \frac{(2\pi i )^n}{n!} (P'_{X,M,T}(t))^j \right) \left(\left(P'_{X,M,T}(t)\right)^2-\left(P'_{X,M,T}(t)\right)\right)h-h\|\\
%		&=& \| h + \sum_{n=2}^N\left( \sum_{j=0}^{n-2} \frac{(2\pi i )^n}{n!} (P_{X,M,T}(t))^j \right) \left(\left(P_{X,M,T}(t)\right)^2-\left(P_{X,M,T}(t)\right)\right)h-h\|\\
%		&=& \|h + \sum_{n=2}^N\left( \sum_{j=0}^{n-2} \frac{(2\pi i )^n}{n!} (P_{X,M,T}(t))^j \right) \left(\left(P_{X,M,T}(t)\right)^2-\left(P_{X,M,T}(t)\right)\right)h-e^{2\pi i P_{X,M,T}}h  \|  \\
%		&\leq& \epsilon\|h\|.
%	\end{eqnarray*}
\begin{eqnarray*}
		&&\|u'_T(t)h - h\| \\
		&=& \|u'_T(t)h-e^{2\pi i \mathcal{H}(P'_{X,M,T})(t)}h\|\\
        &\leq & \| u'_T(t)h-e^{2\pi i P'_{X,M,T}(t)}h\|+\|e^{2\pi i P'_{X,M,T}(t)}h-e^{2\pi i \mathcal{H}(P'_{X,M,T})(t)}h\|\\
&\leq & \epsilon +500000\epsilon \leq 1000000\epsilon.
	\end{eqnarray*}

	By the same reason, for any $h \in L^2 (X_{[80 T\delta, 95 T\delta]},S)$, we have
	\[
	\|u'_T(t)h - h\|= \|u'_T(t)h- e^{2\pi i P'_{X,M,T}(t)}h\|\leq 1000000\epsilon.
	\]
	Finally, for all $h \in L^2 (X_{[95 T\delta, 300 T\delta]},S)$, $u'_T(t)h$ belongs to $ L^2 (X_{[90 T\delta, 300 T\delta]},S)$.
	The above computation shows that with respect to the decomposition
	\[
	L^2 (X_{\leq 300 T\delta},S) = L^2 (X_{\leq 90 T\delta},S) \oplus L^2 (X_{[90 T\delta, 300T\delta]},S),
	\]
	the path $u'_T(t)$ can be perturbed to the following path of invertible operators
	\[
	\begin{bmatrix}
	I & 0 \\
	0 & u_{X,M,T}(t)
	\end{bmatrix},
	\]
	such that
	\[
	\|u'_T(t) - \begin{bmatrix}
	I & 0 \\
	0 & u_{X,M,T}(t)
	\end{bmatrix}   \|\leq 4 \times 1000000 \epsilon.
	\]
	Hence we have
	\begin{equation}\label{eq:uXMT}
		\partial (\ind_{X,M,T}(D)) = [ u_{X,M,T}(\cdot) ],
	\end{equation}
	where the path $u_{X,M,T}(t)$ is supported on $X_{[90T\delta, 300T\delta]}$ (see Figure \ref{fig:support rho representative}). The proof is then complete.
	\epr
	
	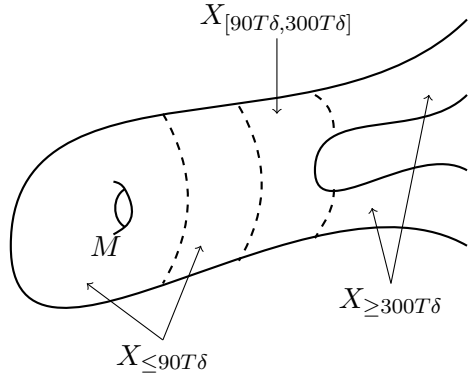
\begin{figure}[h]
		\centering
		\begin{tikzpicture}[scale=0.5]
		\draw[thick](18,6)to[out=225,in=90](6,0)to[out=270,in=150](18,0);
		\draw[thick](18,4)to[out=225,in=90](14,2)to[out=270,in=150](18,2);
		\draw[thick](9,1.5)to[out=210,in=150](9,0.5);
		\draw[thick](8.7,1.7)to[out=350,in=120](9,1.5)to[out=300,in=40](9,0.5)to[out=220,in=30](8.7,0.3);
		\draw[dashed,thick](10,3.5)to[out=300,in=60](10,-1);
		\draw[dashed,thick](12,3.7)to[out=300,in=60](12,-0.4);
		\draw[dashed,thick](14,4)to[out=330,in=450](14.5,3);
		\draw[dashed,thick](14.5,1.5)to[out=270,in=60](14,0.2);
		\node at (8.5,0) {$M$};
		\node at (10,-3) {$X_{\leq 90T\delta}$};
		\draw[->](10,-2.5)--(8,-1);
		\draw[->](10,-2.5)--(11,0);
		\node at (16,-1.5) {$X_{\geq 300T\delta}$};
		\draw[->](16,-1)--(15.5,1);
		\draw[->](16,-1)--(17,4);
		\node at(13,6){$X_{[90T\delta, 300T\delta]}$};
		\draw[->](13,5.5)--(13,3.5);
		\end{tikzpicture}
		\caption{for all $t\in[0, \infty)$, operator $u_{X,M,T}(t)$ is supported on $X_{[90T\delta, 300T\delta]}$.}
		\label{fig:support rho representative}
	\end{figure}
	
	%In the following, without loss of generality, we assume that $X_{250T\delta, 350T\delta}$ is a tubular neighbourhood of $X_{300T\delta}$ and that $X_{300T\delta}$ is an $n-1$ dimensional submanifold.
	
	\begin{Def}\label{def:higher rho at infinity}
		For any $T$ sufficiently large, the element in $K_1(C^*_{L,0}(X_{[90T\delta, 300T\delta]})^G)$ represented by  	
		$u_{X,M,T}(\cdot)$ (defined in Equation \eqref{eq:uXMT} ) is the higher rho invariant at infinity, denoted by $\rho_{X,M,T}(D)$.
	\end{Def}
	
	%As $T$ tends to infinity, the distance of a base point $x_0$ and the support of $\rho_{X,M,T}(D)$, $X_{[90T\delta, 300T\delta]}$, tends to infinity. That is why $\rho_{X,M,T}(D)$ is called the higher rho invariant at infinity.

The reason why we call it the higher rho invariant at infinity, is that we can push its support as far as possible from any base point of $X$.

Certainly, the obstruction algebra $C^*_{L,0}(X_{[90T\delta, 300T\delta]})^G$ and the $K$-group $K_1(C^*_{L,0}(X_{[90T\delta, 300T\delta]})^G)$ depend on $T$, so does the higher rho invariant $\rho_{X,M,T}(D)$. However, the dependence is actually not a problem  for Definition \ref{def:higher rho at infinity}. In fact, consider $T<T'$, under the following two embeddings
\begin{eqnarray*}
&\iota: K_1(C^*_{L,0}(X_{[90T\delta, 300T\delta]})^G)\to  K_1(C^*_{L,0}(X_{\leq 300T'\delta})^G)\\
&\iota': K_1(C^*_{L,0}(X_{[90T'\delta, 300T'\delta]})^G)\to  K_1(C^*_{L,0}(X_{\leq 300T'\delta})^G),
\end{eqnarray*}
it follows directly from the definition that
\[
\iota (\rho_{X,M,T}(D))=\iota'(\rho_{X,M,T'}(D)).
\]

Different from the classical higher rho invariant, the higher rho invariant at infinity is an obstruction of a manifold having a metric with uniform positive scalar curvature. More precisely, if $X$ bears a metric with uniform positive scalar curvature, then $\rho_{X, M, T}(D)$ is trivial. Moreover, the higher rho invariant at infinity is designed to differentiate  uniform  positive scalar curvature metrics at infinity. We will discuss this point in a forthcoming paper.

	\subsection{Delocalized eta invariant at infinity}\label{subsec: define delocalized eta}

	%\begin{Prop}
	%Let $g$ be a nontrivial element of $G$, we have
	%	\[
	%	\tr_g(\ind_G D) = \tau_g (\rho(D, T) ),
	%	\]
	%	for any $T$.
	%\end{Prop}
	%
	%\bpr
	%
	%
	%
	%As mentioned above, we have $\tr_g (\ind_G D) = \tr_g (\ind_c(D, T))$ for any $T$. Assume that $\ind_c(D, T)$ is represented by projection $p_T$. Let the path $a_T(t)$ be a lift of $p_T$ in $C^*_L(X_{\leq 300T\delta})^G$, such that $a_T(0)= p_T$. Then $\rho(D, T) = \partial (\ind_c(D, T)) = e^{2\pi i a_T(\cdot)}$. It follows that
	%\[
	%\tau_g (\rho(D, T) ) = \int_0^\infty \tr_g ( (e^{2\pi i a_T(t)})' e^{-2\pi i a_T(\cdot)} ) dt =  \int_0^\infty \tr_g ( a_T'(t) ) dt.
	%\]
	%Since $g$ is nontrivial element and the propagation of $a_T'(t)$ tends to zero as $t$ tends to infinity, we have
	%\[
	% \int_0^\infty \tr_g ( a_T'(t) ) dt = \tr_g (a_T(0)) = \tr_g (P_T).
	%\]
	%The proof is completed.
	%
	%\epr
	
	In this subsection, we define the delocalized eta invariant at infinity, which measures the nonlocality of higher index of the Dirac operators on noncompact spin manifolds endowed with a metric admitting positive scalar curvature at infinity numerically.
	
	Consider the connecting map
	\[
	\partial: K_0(C^*_r(G)) \to K_1(C^*_{L,0}(X_{[90T\delta, 300T\delta]})^G)
	\]
	in the $K$-theory six-term exact sequence associated to the short exact sequence
	\[
	0\to C_{L,0}^*(X_{[90T\delta, 300T\delta]})^G \to C_{L}^*(X_{[90T\delta, 300T\delta]})^G \to C^*(X_{[90T\delta, 300T\delta]})^G\to 0.
	\]
	
	Here $K_0(C^*_r(G))$ is identified with $K_0(C^*(X_{[90T\delta, 300T\delta]})^G)$. The following lemma is similar to Lemma 3.9 of \cite{XYdelocalizedetainvalgebraicityandKtheoryofgroupCalgebras}, which shows that the delocalized trace on $K_0(C^*_r(G))$, defined in line \eqref{eq:orbitaltracemap}, and the determinant map, defined in line \eqref{eq:determinantmap}, are compatible with the $K$-theory boundary map. Note that although $X_{[90T\delta, 300T\delta]}$ is not a complete manifold, it can be embedded into a complete manifold with proper, free and cocompact $G$-action. For example, let $N\subset X$ be a manifold with boundary which contains $X_{\leq 400T\delta}$, then a perturbation around $\partial N$ makes $N \cup N^{op}$ into a complete smooth manifold, in which one can embed  $X_{[90T\delta, 300T\delta]}$. Thus for any nontrivial element $g\in G$ whose conjucgay class has polynomial growth, the determinant map
	\[
	\tau_g : K_1(C_{L,0}^*(X_{[90T\delta, 300T\delta]})^G)\to \mathbb{C}
	\]
	is still well defined.
	
	\begin{Lem}\label{lem commut of tr and tau}
		Let $g\in G$ be a nontrivial element whose conjugacy class has polynomial growth, then the following diagram commutes:
		\[
		\xymatrix{
			K_0(C^*_r(G)) \ar[r]^{\partial\ \ \ \ \ \ \ \ \ }  \ar[d]_{\tr_g } &K_1(C_{L,0}^*(X_{[90T\delta, 300T\delta]})^G) \ar[d]^{-\tau_g }\\
			\mathbb{C} \ar[r] &\mathbb{C}.
		}
		\]
		
	\end{Lem}
	
	\bpr
	For any element $[p]\in K_0(C^*_r(G))\cong K_0(C^*(X_{[90T\delta, 300T\delta]})^G)$, $\partial [p] =[u(\cdot)]$ is defined by
	\[
	u(t)= e^{2\pi i a(t)}, t\in [0, \infty),
	\]
	where $a(t)$ is a lift of $p$ in $C_{L}^*(X_{[90T\delta, 300T\delta]})^G$ with $a(0)= p$. By definition, we have
	\[
	\tau_g (u)= \frac{1}{2\pi i }\int_0^\infty \tr_g(u(t)^{-1}u'(t)) dt = \int_0^\infty \tr_g(a'(t)) dt=\lim\limits_{t\to \infty}\tr_g(a(t))-\tr_g(a(0)).
	\]
	Since $g$ is not the identity element, $\tr_g (a(t))=0$ when $t$ is sufficiently large such that the propagation of $a(t)$ is small enough. It follows that
	\[
	\tau_g (\partial [p])=\tau_g([u])= -\tr_g ([p]).
	\]
	The lemma is then proved.
	\epr
	
	By definition, Lemma \ref{lem commut of tr and tau} implies that
	\[\tr_g(\ind_G (D)) = -\tau_g (\rho_{X,M,T}(D) ).\]
	
	Inspired by Theorem 4.3 of \cite{XYdelocalizedetainvalgebraicityandKtheoryofgroupCalgebras}, we define the delocalized higher eta invariant at infinity as follows.
	
	\begin{Def}\label{def: delocalized eta}
		Let $g\in G$ be a nontrivial element. The delocalized eta invariant at $g\in G$ at infinity, denoted by $\eta_{g, \infty}(D)$, is defined to be
		\[
		\eta_{g, \infty}(D):=-2\tr_{g} (\ind_G (D))=2\tau_g (\rho_{X,M,T}(D)).
		\]
	\end{Def}
	\vspace{.2cm}

	\subsection{A formula for the higher index}\label{sec:formula for higher index}
	
	In this subsection, we establish a formula for the higher index of the Dirac operator we considered above for the purpose of computing the delocalized eta invariant defined in Definition \ref{def: delocalized eta}. In Subsections \ref{sec:formula for higher index} and \ref{subsec:formulafordelocalizedeta}, we assume in addition that $X$ has bounded geometry, and consider only finitely generated discrete group. This is due to a special choice of parametrix of the Dirac operator.
	
	Compare Subsection \ref{sec:formula for higher index} and \ref{subsec:formulafordelocalizedeta} with \cite[Section 5]{HWW19equivariantAPSindexforproperaction}.
	
	%Let $X$ be a complete Riemannian manifold (even dimensional) with a spin structure on which a discrete group $G$ acts freely, properly and preserving the metric.
	%Let $M\subset X$ be a $G$-invariant submanifold with compact quotient $M/G.$
	%Assume that the complement $X\backslash M$ admits a metric of uniform positive scalar curvature $h$.
	%
	%Let $D$ be the spin Dirac operator associated to the metric $h$ on $X$. Denote by $D_M, D_c$ the restriction of $D$ to $M$ and $C=X\backslash M$ respectively.
	Consider the Dirac operator on $X$ (even dimensional) which is odd with respect to the grading $L^2(X, S^+)\oplus L^2(X, S^-)$:
	\[
	D=\begin{bmatrix}0 & D^- \\ D^+ & 0\end{bmatrix} \qquad (D^+)^*=D^-.
	\]
	
	Set $Q(t):=\frac{1-e^{-tD^-D^+}}{D^-D^+}D^-, t\geq 0$. Operators $Q(t)$ are well defined since $X$ is a complete manifold, $D^-D^+$ admits a unique self-adjoint extension and the functional calculus can be applied .  Recall $D_c$ is the restriction of $D$ to $X\slash M$. Choose a parametrix for $D_c^+$ to be its inverse $Q_c:=(D_c^-D_c^+)^{-1}D_c^{-}.$ Operators $Q_c$ are well defined since $D_c^-D_c^+$ is bounded below by a positive number and admits a unique Friedrich's extension (a self-adjoint extension).
	
	Without loss of generality, we assume that $M$ is a manifold with boundary $\partial M$, and assume that
	\[(-1, 1)\times \partial M
	= \{x\in X, \text{d}(x, \partial M) \leq 1\}
	\]
	is a tubular neighbourhood of $\partial M$ in $X$. %For any real number $a,\  b\in [0,1]$, we simply denote $(a,b)\times \partial M $ to be the set
	%\[
	%\{x\in X, a< \text{d}(x, M) < b\}.
	%\]
	Let $\psi_1$ be a smooth function from $X$ to $[0,1]$ such that $\psi_1 \equiv 1 $ on $X_{[0, \frac{1}{4}]}$ and $\psi_1 \equiv 0 $ on $X_{\geq \frac{3}{4}} $. Set $\psi_2= 1- \psi_1$. Let $\varphi_1,\ \varphi_2$ be smooth functions from $X$ to $[0,1]$ such that $\varphi_1 \equiv 1$ on $X_{[0, \frac{7}{8}]}$ and $\varphi_1 \equiv 0 $ on $X_{\geq 1}$, while $\varphi_2 \equiv 0$ on $X_{[0, \frac{1}{16}]}$ and $\varphi_2 \equiv 1$ on $X_{\geq \frac{1}{8}}$. Note that $\varphi_j \psi_j =\psi_j$ and $[D,\varphi_j]\psi_j=0$, $j=1,2$.
	
	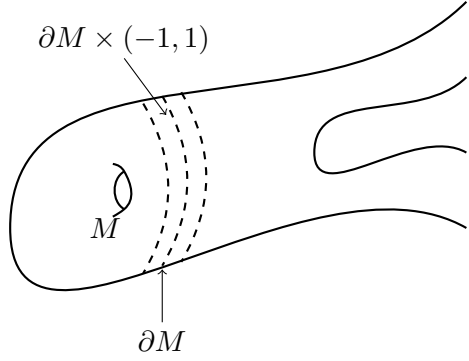
\begin{figure}[h]
		\centering
		\begin{tikzpicture}[scale=0.5]
		\draw[thick](18,6)to[out=225,in=90](6,0)to[out=270,in=150](18,0);
		\draw[thick](18,4)to[out=225,in=90](14,2)to[out=270,in=150](18,2);
		\draw[thick](9,1.5)to[out=210,in=150](9,0.5);
		\draw[thick](8.7,1.7)to[out=350,in=120](9,1.5)to[out=300,in=40](9,0.5)to[out=220,in=30](8.7,0.3);
		\draw[dashed,thick](10,3.5)to[out=300,in=60](10,-1);
		\draw[dashed,thick](9.5,3.3)to[out=300,in=60](9.5,-1.2);
		\draw[dashed,thick](10.5,3.6)to[out=300,in=60](10.5,-0.9);
		\node at (8.5,0) {$M$};
		\node at (10,-3){$\partial M$};
		\draw[->](10,-2.5)--(10,-1.1);
		\node at (9,5) {$\partial M\times (-1,1)$};
		\draw[->](9,4.5)--(10.1, 3);
		%\node at (13,2) {$X\backslash M$};
		%\node at (0,0) {$X$};
		%\draw [->](4,0)--(5.5,0);
		\end{tikzpicture}
		\caption{Tubular neighbourhood of $\partial M$.}
		\label{fig:tubular partial M}
	\end{figure}

	We construct parametrices for $D^+$ as follows,
	\begin{eqnarray*}
		R&=&\varphi_1 Q\psi_1+\varphi_2 Q_c\psi_2, \\
		R'&=& \psi_1 Q \varphi_1+\psi_2 Q_c \varphi_2.
	\end{eqnarray*}
	Denote
	\begin{eqnarray*}
		& S_0=1-RD^+,\\
		& S_1=1-D^+R,\\
		& S_0' = 1- R'D^+.
	\end{eqnarray*}
	
	Denote by $C^*(X, M)^G$ the equivariant Roe algebra supported near $M.$
	
	The proof for the following Lemma goes verbatim as the one for Lemma 4.5 in  \cite{HWW19equivariantAPSindexforproperaction}.
	\begin{Lem}
		The operators $S_0, S_1, S'_0$ have smooth kernels and
		\[
		S_0, S_1, S_0'\in C^*(X, M)^G.
		\]
	\end{Lem}

	%The higher index
	%\[
	%\ind_G: K_1(M(C^*(X, M)^G)/C^*(X, M)^G)\rightarrow K_0(C^*(X, M)^G)\cong K_0(C^*_r(G))
	%\]
	%is defined to be the $K$-theoretic boundary map associated to the short exact sequence
	%\[
	%0\rightarrow C^*(X, M)^G\rightarrow \mathcal M(C^*(X, M)^G)\rightarrow \mathcal M(C^*(X, M)^G)/C^*(X, M)^G\rightarrow 0.
	%\]
	
	Furthermore, by the same argument in \cite{HWW19equivariantAPSindexforproperaction}, one can show that $S_0^2$, $S_1^2$, $(S_0')^2$ and $S_0$, $S_1$, $S'_0$ are all $g$-trace classes with smooth Schwartz kernel when the conjugacy class of $g$ has polynomial growth, and when the manifold $X$ is assumed to have bounded geometry. Note that we have assumed that the positive scalar curvature outside compact subset $M$ is uniformly bounded below.
	
The following proposition has been proven by Hochs, Wang, and Wang in \cite[Lemma 5.2]
{HWWanequiapsindthemforpropacII}. For the sake of the completion, we still sketch the idea of the proof below.
	\begin{Prop}\label{prop:formula for higher index}
		The higher index for $D$ is given by
		%\[
		%\ind_G D:=\partial\begin{bmatrix}0 & R\\ D^+ & 0\end{bmatrix}=\begin{bmatrix}S_0^2 & S_0(S_0+1)R \\ S_1 D^+ & 1-S_1^2\end{bmatrix}-\begin{bmatrix}0 & 0 \\ 0 & 1\end{bmatrix}
		%\]
		\[
		\ind_G (D)=\begin{bmatrix}S_0^2 & S_0(S_0+1)R \\ S_1 D^+ & 1-S_1^2\end{bmatrix}-\begin{bmatrix}0 & 0 \\ 0 & 1\end{bmatrix}
		\]
	\end{Prop}
	\bpr
	This can be seen by choosing the real function $b: (-\infty, \infty)\to [-1,1]$ in Subsection \ref{sec:higher index a definition} to be \[
	b(x)= \left\{\begin{array}{cc}
	\sqrt{1- f^2(x)} & x\geq 0 \\
	-\sqrt{1- f^2(x)} & x<0
	\end{array}
	\right.,
	\]
	where $f: (-\infty, \infty)\to [-1,1] $ is a real function defined as
	\[
	f(x)= \left\{\begin{array}{cc}
	0 & x\geq \frac{\sqrt{h_0}}{2} \\
	1-\frac{2x}{\sqrt{h_0}} & x\in [0,\frac{\sqrt{h_0}}{2}] \\
	1+ \frac{2x}{\sqrt{h_0}} & x\in [-\frac{\sqrt{h_0}}{2}, 0]\\
	0 & x<-\frac{\sqrt{h_0}}{2}
	\end{array}
	\right..
	\]
	\epr

	\subsection{Formula for the delocalized eta-invariant at infinity}\label{subsec:formulafordelocalizedeta}
	
	%\begin{Def}
	%Let $g\in G$ be a nontrivial element. The delocalized eta-invariant at $g\in G$ is
	%\[
	%\eta_{g, \infty}(h)=\lim_{t\to 0}2\int_t^{\infty}\tr_g(e^{-sD_c^-D_c^+}D_c^-\psi_2')ds.
	%\]
	%\end{Def}
	
	The main goal of this subsection is to  prove a formula for the delocalized eta invariant at infinity  in light of the techniques developed in \cite{HWW19equivariantAPSindexforproperaction}, along with which we establish an Atiyah-Patodi-Singer type index theorem for Dirac operator on manifold with a metric with uniform positive scalar curvature at infinity. We mention that with the technique developed in \cite{WangWang16jdg}, one can generalize the main result of this subsection to the case of proper actions.
	
%	Let $\sigma$ be the grading operator on the spinor bundle of $X$.
	%The well-definedness of $\eta_{g,\infty}$ follows from the following APS-index theorem.
	
	\begin{Thm}\label{thm:formula for delocalized eta}
		Let $g$ be an element of $G$ whose conjugacy class has polynomial growth. Then
		\begin{equation}\label{eq:main eq 1}
			\tr_{g}(\ind_G (D))=\int_{M^g}I(g)-\lim_{t\to 0}\int_t^{\infty}\tr_g(e^{-sD_c^-D_c^+}D_c^- [D^+,\psi_2])ds,
		\end{equation}
		where $M^g$ is the fixed point submanifold  of $g$, and $I(g)$ is the integrand as in the Atiyah-Segal-Singer fixed point theorem.
		In particular, if $g\neq e$, we have
		\begin{equation}\label{eq:main eq 2}
			\frac{1}{2}\eta_{g, \infty}(D)=-\tr_{g}(\ind_G (D))=\lim_{t\to 0}\int_t^{\infty}\tr_g(e^{-sD_c^-D_c^+}D_c^-  [D^+,\psi_2])ds.
		\end{equation}
		The second integral on the right hand side of \eqref{eq:main eq 1}
		and the integral on the right hand side of \eqref{eq:main eq 2} are independent from the choice of the cutoff function $\psi_2.$
	\end{Thm}
	
	We divide the proof of Theorem \ref{thm:formula for delocalized eta} into the following series of lemmas.
	
	%\begin{Lem}\label{lem:basic trace formula for index}
	%	Recall that $S^2_0$ and $S^2_1$ are both $g$-trace classes, then
	%	\[
	%	\tr_{g,*} (\ind_G D)= \tr_g (S_0^2) -\tr_g (S^2_1).
	%	\]
	%\end{Lem}
	%\bpr
	%This lemma follows basic index theory.
	%\epr

	\begin{Lem}\label{lem:expanse s0s1}
		We have
		\begin{eqnarray}
			S_0 & = & \phi_1 (1- QD^+) \psi_1 + \phi_1 Q  [D^+,\psi_1]  + \phi_2 Q_c [D^+,\psi_2] ,\\
			S_0' & = & \psi_1 (1- QD^+)   \phi_1 + \psi_1 Q  [D^+,\phi_1]  + \psi_2 Q_c [D^+,\phi_2] ,\\
			S_1 & = & \phi_1  (1- D^+ Q)  \psi_1 - [D^+,\phi_1]  Q  \psi_1  - [D^+,\phi_2] Q_c  \psi_2 .
		\end{eqnarray}
	\end{Lem}
	\bpr
	Direct computation.
	\epr
	
	From now on, to emphasize the dependence of operators $Q$ (resp. $Q_c,\ S_0,\ S_0',\ S_1$) on $t$, we write $Q(t)$ (resp. $Q_c(t),\ S_0(t),\ S_0'(t),\ S_1(t)$) for $Q$ (resp. $Q_c,\ S_0,\ S_0',\ S_1$).
	
	\begin{Lem}\label{lem:eliminate square}
		We have
		\[
		\tr_g (S_0^2)- \tr_g (S_1^2)= \tr_g (S_0)-\tr_g(S_1).
		\]
	\end{Lem}
	\bpr
	To begin with, note that
	\[
	\tr_g (S_0)-\tr_g(S_0^2)
	= \tr_g(S_0-S_0^2)
	= \tr_g(S_0 R D^+).
	\]
	By definition, $S_0 R D^+= R S_1 D^+ .$ Moreover, applying trace property of $\tr_g$, we obtain
	\[\tr_g (R S_1 D^+)
	= \tr_g (S_1 D^+ R)
	= \tr_g (S_1(1-S_1))
	= \tr_g (S_1)- \tr_g(S_1^2).\]
	Hence the lemma follows.
	\epr
	
	\begin{Lem}\label{lem:contribution of inner}
		\[  \lim\limits_{t\to 0^+}[	\tr_g(S_0') - \tr_g (S_1) ]=  \int_{M^g} I(g) \]
	\end{Lem}
	\bpr
	By Lemma \ref{lem:expanse s0s1}, one can see
	\[
	\tr_g (S_0')  =  \tr_g ( \psi_1  (1- QD^+)   \phi_1 + \psi_1 Q  [D^+,\phi_1]  + \psi_2 Q_c [D^+,\phi_2]  ).
	\]
	However, since $[D^+,\phi_1]\psi_1 = [D^+,\phi_2]\psi_2=0$,
	\[
	\tr_g (S_0') =  \tr_g ( \psi_1  (1- QD^+)  \phi_1) =\tr_g ((1- QD^+)  \psi_1).
	\]
	Similarly,
	\[
	\tr_g (S_1) = \tr_g ((1- D^+Q) \psi_1).
	\]
	%Choose $\psi_1$ properly, such  that
	%\begin{eqnarray*}
	%&& 2\tr_g ((1- QD^+) )= \tr_g (\widetilde S_0\psi_1)+ \tr_g (\widetilde S_0 (1_{\widetilde M}-\psi_1)) = 2 \tr_g (\widetilde S_0\psi_1)\\
	%&& 2\tr_g ((1- QD^+) )= \tr_g (\widetilde S_1\psi_1)+ \tr_g (\widetilde S_1 (1_{\widetilde M}-\psi_1)) = 2 \tr_g (\widetilde S_1\psi_1).
	%\end{eqnarray*}

	Then a standard heat kernel argument shows that
	\[
	\lim\limits_{t\to 0^+} [\tr_g ( S'_0)-\tr_g ( S_1)]=  \int_{M^g}I(g).
	\]
	This proves the Lemma.
	\epr

	\begin{Lem}\label{lem:contribution of boundary}
		\[\lim\limits_{t\to 0^+}[\tr_g ( S_0) -\tr_g (S'_0)] =- \lim_{t\to 0^+}\int_t^{\infty}\tr_g(e^{-sD_c^-D_c^+}D_c^-  [D^+,\psi_2])ds\]
	\end{Lem}
	
	\bpr
	By Lemma \ref{lem:expanse s0s1}, the following equality
	\[
	\tr_g (S_0) -\tr_g (S'_0) = \tr_g( \phi_1 Q  [D^+,\psi_1] ) + \tr_g (\phi_2 Q_c  [D^+,\psi_2])
	\]
	holds.
	Construct a new parametrix for $D_c$
	\[
	Q_c':=\frac{1- e^{-tD_c^-D_c^+}}{D^{-}_cD^+_c}D_c^-.
	\]
	Choose $\phi\in C^\infty (X)$ such that $\phi=1$ on the support of $[D^+,\psi_j], j=1,2$ and $\phi=0$ on the support of $1-\phi_j$. This implies that
	\begin{eqnarray}
		&& \tr_g ( \phi_1 Q  [D^+,\psi_1] -  \phi_2 Q_c'   [D^+,\psi_1])\\
		&=& \tr_g ( \phi (Q-Q_c')  [D^+,\psi_1]) +\tr_g  ((1-\phi)( \phi_1 Q -\phi_2 Q_c' )[D^+,\psi_1]) \\
		&=& \tr_g ( \phi (Q-Q_c') [D^+,\psi_1]) .
	\end{eqnarray}

	%Note that $\tr_g( \phi_1 Q(t)   [D^+,\psi_1] ) $ and $\tr_g( \phi_2 Q_c'(t)  [D^+,\psi_1] ) $ do not depend on $t$.
	Hence we have
	\begin{eqnarray}
		&& \tr_g ( \phi (Q-Q_c')  [D^+,\psi_1])\\
		&=& \tr_g ( \phi( Q(t_0)-Q_c'(t_0) ) [D^+,\psi_1])\\
		&=& \tr_g ( \phi \int_0^{t_0}  e^{-s D^-D^+}D^- - e^{-s D_c^- D_c^+}D_c^-  ds [D^+,\psi_1] )\\
		&=&  \tr_g (  \int_0^{t_0}\phi (  e^{-s D^-D^+}D^- - e^{-s D_c^- D_c^+}D_c^- ) [D^+,\psi_1] ds  ).
	\end{eqnarray}
	However, by a finite propagation argument, one can show
	\[\phi(e^{-s D^-D^+}D^- - e^{-s D_c^- D_c^+}D_c^- ) [D^+,\psi_1]=0\]
	when $s$ is sufficiently small. Thus
	\[\tr_g ( \phi_1 Q  [D^+,\psi_1] -  \phi_2 Q_c'  [D^+,\psi_1])=\tr_g ( \phi (Q-Q_c' ) [D^+,\psi_1])=0. \]
	On the other hand, we have
	\begin{eqnarray}
		&&\lim\limits_{t\to 0^+}[\tr_g (\phi_2 Q_c  [D^+,\psi_2] - \phi_2 Q_c'  [D^+,\psi_2])]\\
		&=& \lim\limits_{t\to 0^+}[\tr_g ( \phi_2  (Q_c - \frac{1- e^{-tD_c^-D_c^+}}{D^{-}_cD^+_c}D_c^-)   [D^+,\psi_2]  )]\\
		&=&\lim\limits_{t\to 0^+}[ \tr_g ( \phi_2  (Q_c - \frac{1- e^{-tD_c^-D_c^+}}{D^{-}_cD^+_c}D_c^-D_c^+ Q_c)   [D^+,\psi_2]  )]\\
		&=&\lim\limits_{t\to 0^+}[ \tr_g ( \phi_2   e^{-tD_c^-D_c^+} Q_c   [D^+,\psi_2]  )]\\
		&=& -\lim\limits_{t\to 0^+}[\tr_g (\phi_2\int_t^\infty  e^{-sD_c^-D_c^+} D^-_c dt   [D^+,\psi_2])]\\
		&=& -\lim\limits_{t\to 0^+}\int_t^\infty \tr_g (e^{-sD_c^-D_c^+} D^-_c  [D^+,\psi_2]) dt .
	\end{eqnarray}
	Note that $[D^+,\psi_1]+[D^+,\psi_2]=0$, thus we have
	\begin{eqnarray*}
		&&\lim\limits_{t\to 0^+}[\tr_g (S_0) -\tr_g (S'_0)] \\
		& =& \lim\limits_{t\to 0^+}[\tr_g( \phi_1 Q   [D^+,\psi_1] ) + \tr_g (\phi_2 Q_c [D^+,\psi_2])]\\
		&=&  \lim\limits_{t\to 0^+}[[\tr_g( \phi_1 Q   [D^+,\psi_1] )-\tr_g (\phi_2 Q_c'  [D^+,\psi_1])] \\
		&& +\lim\limits_{t\to 0^+}[\tr_g (\phi_2 Q_c   [D^+,\psi_2]) - \tr_g (\phi_2 Q_c'  [D^+,\psi_2]) ]\\
		&=& -\lim_{t\to 0^+}\int_t^\infty \tr_g (e^{-tD_c^-D_c^+} D^-_c  [D^+,\psi_2]) dt.
	\end{eqnarray*}
	The lemma is then proved.
	\epr
	
   The following lemma is a direct result of Proposition 4.1 and Proposition 5.1 of \cite{HWWanequiapsindthemforpropacII}.
	\begin{Lem}\label{lem: basic index theo lemma}
		By the isomorphism between $C^*_r(G)\otimes \mathcal{K}$ and $C^*(X, M)^G$, and basic index theory, we have
		\[
		\tr_{g} (\ind_G D) = \tr_g (S_0^2) -\tr_g (S^2_1).
		\]
	\end{Lem}
	
	\begin{proof}[Proof of Theorem \ref{thm:formula for delocalized eta} ]
		Combine Lemma \ref{lem:eliminate square}, \ref{lem:contribution of inner}, \ref{lem:contribution of boundary} and \ref{lem: basic index theo lemma}, we obtain
		\begin{eqnarray}
			\tr_{g} (\ind_G (D)) &=& \tr_g (S_0^2) -\tr_g (S^2_1)\\
			&=& \tr_g (S_0) -\tr_g (S_1)\\
			&=& \tr_g (S_0) -\tr_g(S_0') + \tr_g(S_0')-\tr_g(S_1)\\
			&=& \int_{M^g}I(g) - \lim_{t\to 0}\int_t^\infty \tr_g (e^{-sD_c^-D_c^+} D^-_c  [D^+,\psi_2]) dt.
		\end{eqnarray}
		The proof is completed.
	\end{proof}

%	that the metric on $X$ is of bounded geometry, and

		\section{M\"uller's Atiyah-Patodi-Singer type theorem}\label{sec manifold with corner}
	
		In this section, we apply the technique from Subsection \ref{subsec:formulafordelocalizedeta} to develop a Atiyah-Patodi-Singer type theorem on manifold with corner of codimension 2. Our computation can be easily generalized to the case of manifolds with corners of codimensions $k\geq 2$. This subsection can be viewed as an $L^2$ generalization of Muller's extraordinary pioneer work in \cite{Muller94manifoldwithboundary} and \cite{Muller96manifoldwithcorner} in the particular case of manifolds with metric with uniform positive scalar curvature at  infinity.

Let us first fix some notions.  Let $M$ be an even dimensional manifold with corner with $\partial M= \partial_1 M \cup \partial_2 M$, where $\partial_1 M$ and $\partial_2 M$ are manifolds with common boundary, i.e. $\partial \partial_1 M= \partial \partial_2 M = Y$. The metric of $M$ is collared near $\partial M$ and have positive scalar curvature on $\partial_1 M$ and $\partial_2 M$. Furthermore, the metric on $\partial_1 M$ and $\partial_2 M$ is collared near $Y$. Let $G$ be a discrete group acting on $M$ properly, freely and cocompactly by isometries.
	
	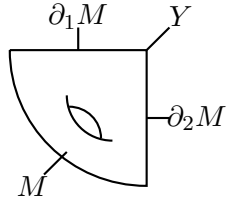
\begin{figure}[h]
		\centering
		\begin{tikzpicture}[scale=0.3]
		\draw[thick](0,3)--(6,3)--(6,-3)to[out=180,in=270](0,3);
		%\draw[thick](3.5,0.5)to[out=315,in=225](7,0.5);
		%\draw[thick](3,0)to[out=50,in=210](3.5,0.5)to[out=30,in=150](7,0.5)to[out=330,in=130](7.5,0);
		\node at (3,4.5){$\partial_1 M$};
		\draw[thick](3,4)--(3,3);
		\node at (8.2,0) {$\partial_2 M$};
		\draw[thick](7,0)--(6,0);
		\node at (7.5,4.5) {$Y$};
		\draw[thick](7,4)--(6,3);
		\node at (1,-3) {$M$};
		\draw[thick](1.5,-2.5)--(2.5,-1.5);
		\draw[thick](2.5,1)to[out=270,in=180](4.5,-1);
		\draw[thick](2.6,0.5)to[out=0,in=90](4,-0.9);
		\end{tikzpicture}
		\caption{Manifold with corner of codimension 2.}
		\label{fig:manifold with corner}
	\end{figure}
	
	Let $Z_i, i=1,2$ be $\partial_i M\cup (Y\times \mathbb{R_+})$ and $Z_{i, \leq T}, i=1,2$ be $\partial_i M\cup (Y\times [0,T])$. Denote $C_i, i=1,2$ as $ \partial_i M \times \mathbb{R}_+, i=1,2$ and $C_{i, \leq T}, i=1,2 $ as $ \partial_i M \times [0,T], i=1,2$. Let $C_0$ be $Y\times \mathbb{R}_+^2$ and $C_{0, \leq T}$ be $Y\times [0,T]\times [0,T]$. Let $M_{i}, i=1,2$ be $M \cup C_i, i=1,2$ and $M_{i, \leq T}, i=1,2$ be $M \cup C_{i, \leq T}, i=1,2$.
	
	\begin{figure}[h]
		\centering
		\begin{tikzpicture}[scale=0.3]
		\draw[thick](0,3)--(6,3)--(6,-3)to[out=180,in=270](0,3);
		\node at (0.8,-3.2) {$M$};
		\draw[thick](1.5,-2.5)--(2.5,-1.5);
		\draw[thick](2.5,1)to[out=270,in=180](4.5,-1);
		\draw[thick](2.6,0.5)to[out=0,in=90](4,-0.9);
		\draw[thick](0,3)--(0,13);
		\draw[thick](6,3)--(6,13);
		\draw[thick](6,-3)--(16,-3);
		\draw[thick](6,3)--(16,3);
		%\node at (10,-4.5) {$C_2=\partial_2 M\times \mathbb{R}_+$};
		%\draw[thick](10,-3.7)--(10,-1);
		\node at (14,-1) {$C_2$($=\partial_2 M\times \mathbb{R}_+$)};
		\node at (-5,8) {$C_1$($=\partial_1 M \times \mathbb{R}_+$)};
		\draw[thick](-0.5,8)--(1,8);
		%\node at (3,8){$C_1$};
		%\node at (3,5){($=\partial_1 M \times \mathbb{R}_+$)};
		\node at (13,10) {$C_0$($=Y \times \mathbb{R}_+^2$)};
		\end{tikzpicture}
		\caption{Complete manifold $X$.}
		\label{fig:complete corner}
	\end{figure}
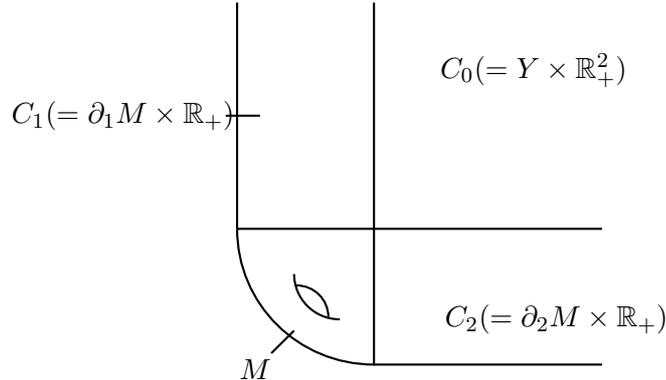
	
	Set a complete manifold
	\begin{equation}\label{eq: def of X}
		X=M_1 \cup ( Z_{2}\times \mathbb{R_+})= M_{2}\cup (Z_{1}\times \mathbb{R_+})
	\end{equation}
	and
	\begin{equation}\label{eq: def of X T}
		X_{\leq T}= M_{1,\leq T}\cup Z_{1,\leq T}\times [0,T]= M_{2,\leq T}\cup Z_{2,\leq T}\times [0,T].
	\end{equation}
	
	We assume in addition  $G$ is a finitely generated discrete group acting on $M$, properly, cocompactly, and freely by isometries.

	Let $\sigma$ be the grading operator on the spinor bundle of $X$.
	
	Recall that $Z_1\times \mathbb{R_+}= C_1\cup C_0$ and $Z_2\times \mathbb{R_+}= C_2\cup C_0$ (see Figures \ref{fig:complete corner} and \ref{fig:complete corner 2}).
	
	\begin{figure}[h]
		\centering
		\begin{tikzpicture}[scale=0.3]
		\draw[thick](0,3)--(6,3)--(6,-3)to[out=180,in=270](0,3);
		\node at (0.8,-3.2) {$M$};
		\draw[thick](1.5,-2.5)--(2.5,-1.5);
		\draw[thick](2.5,1)to[out=270,in=180](4.5,-1);
		\draw[thick](2.6,0.5)to[out=0,in=90](4,-0.9);
		\draw[thick](0,3)--(0,13);
		\draw[thick](6,3)--(6,13);
		\draw[thick](6,-3)--(16,-3);
		\draw[thick](6,3)--(16,3);
		%\node at (10,-4.5) {$C_2=\partial_2 M\times \mathbb{R}_+$};
		%\draw[thick](10,-3.7)--(10,-1);
		\node at (17,-6) {$Z_2\times \mathbb{R}_+$};
		\node at (-5,8) {$Z_1\times \mathbb{R_+}$};
		\draw[dashed,thick](-2,8)--(2,7);
		\draw[dashed,thick](-2,8)--(8,8);
		\draw[dashed,thick](17,-5)--(14,-1);
		\draw[dashed,thick](17,-5)--(14,5);
		%\node at (3,8){$C_1$};
		%\node at (3,5){($=\partial_1 M \times \mathbb{R}_+$)};
		\node at (13,10) {$C_0$($=Y \times \mathbb{R}_+^2$)};
		\end{tikzpicture}
		\caption{$Z_1\times \mathbb{R_+}$ and $Z_2\times \mathbb{R_+}$.}
		\label{fig:complete corner 2}
	\end{figure}
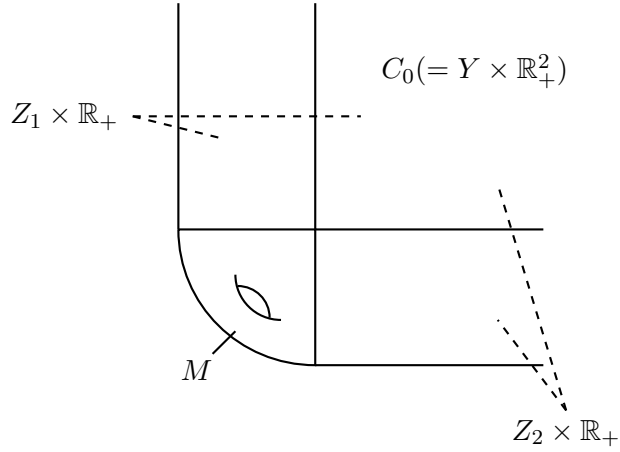
	
	Let $\psi_1:\mathbb{R}_+\to [0,1]$ be a smooth function  such that $\psi_1 \equiv 1 $ on $(0, \frac{1}{4})$ and $\psi_1 \equiv 0 $ on $( \frac{3}{4}, \infty) $. Set $\psi_2= 1- \psi_1$. Let $\varphi_1,\ \varphi_2: \mathbb{R}_+\to [0,1]$ be smooth functions such that $\varphi_1 \equiv 1$ on $(0, \frac{7}{8})$ and $\varphi_1 \equiv 0 $ on $(1,\infty)$, while $\varphi_2 \equiv 0$ on $(0, \frac{1}{16})$ and $\varphi_2 \equiv 1$ on $(\frac{1}{8}, \infty)$. Note that $\varphi_j \psi_j =\psi_j$ and $\varphi_j'\psi_j=0$, $j=1,2$. Denote $\psi_{i,j}$ by $\psi_j\otimes \chi_{Z_i}$, and $\varphi_{i,j}$ by $\varphi_j\otimes \chi_{Z_i}$ for $i,j=1,2$, where $\chi_{Z_i}$ is the characteristic function of $Z_i$. Note that the range of $\varphi_{i,j}$ is contained in $[0,1]$.
	
	Let $D_{X}$ be the Dirac operator on $X$, $D_{Z_i\times \mathbb{R}_+}$ and $D_0$ be the Dirac operator on $Z_i\times \mathbb{R}_+$ and $C_0$ respectively. Let $Q_{i}$ be defined using $(D_{Z_i\times \mathbb{R}_+}^-D_{Z_i\times \mathbb{R}_+}^+)^{-1}D_{Z_i\times \mathbb{R}_+}^-$, the inverse of $D_{Z_i\times \mathbb{R}_+}^+$ , and $Q_0$ be defined using the inverse of $D_0^+$. The existence of $Q_0,\ Q_1,\ Q_2$ are ensured by the uniform positive scalar curvature on all boundary pieces of $M$. For any $t\in [0, \infty)$, set $Q(t)$ as
	\[
	\frac{1-e^{-tD_{X}^-D_X^+}}{D_{X}^-D_X^+}D_{X}^-,
	\]
	$Q'_{i}(t)$ as
	\[
	\frac{1-e^{-tD_{Z_i\times \mathbb{R}_+}^-D_{Z_i\times \mathbb{R}_+}^+}}{D_{Z_i\times \mathbb{R}_+}^-D_{Z_i\times \mathbb{R}_+}^+}D_{Z_i\times \mathbb{R}_+}^-,
	\]
	and $Q'_0(t)$ as
	\[
	\frac{1-e^{-tD_{0}^-D_0^+}}{D_{0}^-D_0^+}D_{0}^-.
	\]
	
	Similarly as in Section \ref{sec:formula for higher index}, we construct two paramatices as following
	\begin{eqnarray*}
		R(t) &=& \varphi_{1,1}\varphi_{2,1} Q(t) \psi_{1,1}\psi_{2,1}+ \varphi_{1,2}Q_1(t) \psi_{1,2}+ \varphi_{2,2}Q_2(t)\psi_{2,2}-\varphi_{1,2}\varphi_{2,2}Q_0(t)\psi_{1,2}\psi_{2,2},\\
		R'(t)&=&  \psi_{1,1}\psi_{2,1} Q(t)  \varphi_{1,1}\varphi_{2,1}+ \psi_{1,2} Q_1(t)  \varphi_{1,2}+ \psi_{2,2}Q_2(t)\varphi_{2,2}-\psi_{1,2}\psi_{2,2}Q_0(t)\varphi_{1,2}\varphi_{2,2}.
	\end{eqnarray*}
	Write
	\begin{eqnarray*}
		S_0(t)&=&1-R(t)D^+_X ,\\
		S_0'(t)&=&1-R'(t)D^+_X,\\
		S_1(t)&=&1-D^+_XR(t).
	\end{eqnarray*}
	From
	\[
	\psi_{1,1}+\psi_{1,2}=1 \text{ and }\psi_{2,1}+\psi_{2,2}=1,
	\]
	we obtain
	\[
	\psi_{1,1}\psi_{2,1}+\psi_{1,2}+\psi_{2,2}-\psi_{1,2}\psi_{2,2}=1.
	\]
	Replacing $1$ in $S_0,\ S_0',\ S_1$ by
	$
	\psi_{1,1}\psi_{2,1}+\psi_{1,2}+\psi_{2,2}-\psi_{1,2}\psi_{2,2},
	$ and direct computation shows that
	\begin{eqnarray*}
		S_0(t)&=& \varphi_{1,1}\varphi_{2,1} e^{-tD_{X}^-D_X^+} \psi_{1,1}\psi_{2,1} + \varphi_{1,1}\varphi_{2,1} Q(t)   [D^+_X,\psi_{1,1}\psi_{2,1}] \\ && + \varphi_{1,2}Q_1(t)   [D^+_X,\psi_{1,2}] + \varphi_{2,2}Q_2(t)   [D^+_X,\psi_{2,2}]  -\varphi_{1,2}\varphi_{2,2}Q_0(t)  [D^+_X,\psi_{1,2}\psi_{2,2}]  ,     \\
		S_0'(t)&=& \psi_{1,1}\psi_{2,1} e^{-tD_{X}^-D_X^+}  \varphi_{1,1}\varphi_{2,1}  +  \psi_{1,1}\psi_{2,1}Q(t)   [D^+_X,\varphi_{1,1}\varphi_{2,1}]  \\&& + \psi_{1,2}Q_1(t)   [D^+_X,\varphi_{1,2}] +\psi_{2,2} Q_2(t)   [D^+_X,\varphi_{2,2}] - \psi_{1,2}\psi_{2,2}Q_0(t)    [D^+_X,\varphi_{1,2}\varphi_{2,2}]   ,             \\
		S_1(t)&=& \varphi_{1,1}\varphi_{2,1} e^{-tD_{X}^-D_X^+} \psi_{1,1}\psi_{2,1} +[D^+_X,\varphi_{1,1}\varphi_{2,1}] Q(t)   \psi_{1,1}\psi_{2,1}\\ && + [D^+_X,\varphi_{1,2}]Q_1(t)    \psi_{1,2} + [D^+_X,\varphi_{2,2}]Q_2(t)   \psi_{2,2} -[D^+_X,\varphi_{1,2}\phi_{2,2})]Q_0(t)    \psi_{1,2}\psi_{2,2}.
	\end{eqnarray*}
	
	Note that $S_0,\ S_0',\  S_1$ and their squares are $e$-trace class operators with smooth Schwartz Kernels. Then taking the trace of the $K$-theoretic index, we obtain
	\begin{equation}
		\tr_e(\ind_G(D_M))=\tr_e (S_0^2(t))-\tr_e (S_1^2(t)).
	\end{equation}
	
	Exactly the same argument of Lemma \ref{lem:eliminate square} shows that
	\begin{equation}
		\tr_e (S_0^2(t))-\tr_e (S_1^2(t))=\tr_e (S_0(t))-\tr_e (S_1(t)).
	\end{equation}
	
	For the same reason as in the proof of Lemma \ref{lem:contribution of inner}, we have
	\[
	\lim\limits_{t\to 0^+}(\tr_e (S'_0(t))-\tr_e (S_1(t)) )= \int_{M} \hat A(M).
	\]

	Similarly as in the proof of Lemma \ref{lem:contribution of boundary}, one can also obtain
	\begin{eqnarray*}
		& & \lim\limits_{t\to 0^+} \tr_e (S_0(t)-S_0'(t))\\
		%&=&
		%&=& \lim\limits_{t\to 0} \tr_e ((Q_1(t)-Q(t))\psi_{1,2}')+ \lim_{t\to 0}\tr_e  ((Q_2(t)-Q(t))\psi_{2,2}')\\&& -\lim_{t\to 0} \tr_e (\varphi_{1,2}\varphi_{2,2}(Q_0(t)-Q(t))(\psi_{1,2}\psi_{2,2})')\\
		&=& \lim\limits_{t\to 0} \tr_e ((Q_1(t)-Q_1'(t))\sigma \psi_{1,2}')+ \lim_{t\to 0}\tr_e  ((Q_2(t)-Q_2'(t))\sigma \psi_{2,2}')\\
		&& -\lim_{t\to 0} \tr_e (\varphi_{1,2}\varphi_{2,2}(Q_0(t)-Q_0'(t))[D^+_X, \psi_{1,2}\psi_{2,2}] )\\
		&=& \lim_{t\to 0} \tr_e ((Q_1(t)-Q_1'(t))\sigma \psi_{1,2}')+ \lim_{t\to 0} \tr_e  ((Q_2(t)-Q_2'(t))\sigma \psi_{2,2}')\\ && -\lim_{t\to0}\tr_e ((Q_0(t)-Q_0'(t))\sigma \psi_{1,2}'\psi_{2,2})-\lim_{t\to 0}\tr_e ((Q_0(t)-Q_0'(t))\sigma \psi_{1,2}\psi_{2,2}').
	\end{eqnarray*}
	Let $\chi_{\psi_{i,2}=1}, i=1,2$ be the characteristic function on the set $\{x\in X| \psi_{i,2}(x)=1\}$, and $\chi_{Z_i\times \mathbb{R}_+}, i=1,2$ be the characteristic function on $Z_i\times \mathbb{R}_+$. Then we have
	\begin{eqnarray*}
		& &\lim\limits_{t\to 0^+} \tr_e (S_0(t)-S_0'(t))\\
		&=& \lim_{t\to 0} \tr_e ((Q_1(t)-Q_1'(t))\sigma \psi_{1,2}' - (Q_0(t)-Q_0'(t))\sigma \psi_{1,2}'\chi_{\psi_{2,2}=1})\\&&+ \lim_{t\to 0}\tr_e  ((Q_2(t)-Q_2'(t))\sigma \psi_{2,2}'-(Q_0(t)-Q_0'(t))\sigma \psi_{2,2}'\chi_{\psi_{1,2}=1})\\&&-\lim_{t\to 0} \tr_e ((Q_0(t)-Q_0'(t))\sigma (\psi_{1,2}'\psi_{2,2}(1-\chi_{\psi_{2,2}=1})
		+\psi_{1,2}\psi_{2,2}'(1-\chi_{\psi_{1,2}=1})))\\
		&=& \lim_{t\to 0} \tr_e ((Q_1(t)-Q_1'(t))\sigma \psi_{1,2}' - (Q_0(t)-Q_0'(t))\sigma \psi_{1,2}'\chi_{Z_2\times \mathbb{R}_+})\\&&+ \lim_{t\to 0}\tr_e  ((Q_2(t)-Q_2'(t))\sigma \psi_{2,2}'-(Q_0(t)-Q_0'(t))\sigma \psi_{2,2}'\chi_{Z_1\times \mathbb{R}_+})\\
		&&-\lim_{t\to 0}\tr_e (\sigma (\psi_{1,2}'\psi_{2,2}(1-\chi_{\psi_{2,2}=1})+\psi_{1,2}\psi_{2,2}'(1-\chi_{\psi_{1,2}=1}))(Q_0(t)-Q_0'(t))\\
		&&-\lim_{t\to 0}\tr_e (\sigma(\psi_{1,2}'(\chi_{Z_2\times \mathbb{R}_+}-\chi_{\psi_{2,2}=1}) + \psi_{2,2}'(\chi_{Z_1\times \mathbb{R}_+}- \chi_{\psi_{1,2}=1}))(Q_0(t)-Q_0'(t)).
	\end{eqnarray*}
	
	\begin{Lem}\label{lem:first term}
		We have
		\begin{enumerate}
			\item \begin{eqnarray*}
				&& [\tr_e ( (Q_1(t)-Q_1'(t))\sigma \psi_{1,2}' - (Q_0(t)-Q_0'(t))\sigma \psi_{1,2}'\chi_{Z_2\times \mathbb{R}_+} )] \\
				&=& -\frac{1}{\sqrt{4\pi }} \int_t^\infty \frac{1}{\sqrt{s}} \tr_e ( e^{-sD_{Z_1}^2} D_{Z_1} - e^{-sD_{Y\times \mathbb{R}_+}^2} D_{Y\times \mathbb{R}_+})ds
			\end{eqnarray*}
			\item
			\begin{eqnarray*}
				&& [\tr_e ((Q_2(t)-Q_2'(t))\sigma \psi_{2,2}' - (Q_0(t)-Q_0'(t))\sigma \psi_{2,2}'\chi_{Z_1\times \mathbb{R}_+})] \\
				&=& -\frac{1}{\sqrt{4\pi }} \int_t^\infty \frac{1}{\sqrt{s}} \tr_e ( e^{-sD_{Z_2}^2} D_{Z_2} - e^{-sD_{Y\times \mathbb{R}_+}^2} D_{Y\times \mathbb{R}_+})ds
			\end{eqnarray*}
		\end{enumerate}
	\end{Lem}
	\bpr It is sufficient to prove the first item. By the proof of Lemma \ref{lem:contribution of boundary}, we have
	\begin{eqnarray*}
		&& [\tr_e ((Q_1(t)-Q_1'(t))\sigma \psi_{1,2}' - (Q_0(t)-Q_0'(t))\sigma \psi_{1,2}'\chi_{Z_2\times \mathbb{R}_+})] \\
		&=& - \int_t^\infty  tr  (e^{-sD_{Z_1\times \mathbb{R}_+}^-D_{Z_1\times \mathbb{R}_+}^+}D_{Z_1\times \mathbb{R}_+}^--e^{-sD_0^-D_0^+}D_0^- \chi_{Z_2\times \mathbb{R}_+} )\sigma \psi_{1,2}'   ds.
	\end{eqnarray*}
	Note that
	\[
	D_{Z_1\times \mathbb{R}_+}=\begin{bmatrix}
	0 & D_{Z_1\times \mathbb{R}_+}^-\\
	D_{Z_1\times \mathbb{R}_+}^+ &0
	\end{bmatrix}= \begin{bmatrix}
	0 & \frac{\partial}{\partial x}+D_{Z_1} \\
	-\frac{\partial}{\partial x}+D_{Z_1} & 0
	\end{bmatrix},
	\]
	and
	\[
	D_{0}=\begin{bmatrix}
	0 & D_{0}^-\\
	D_{0}^+ &0
	\end{bmatrix}= \begin{bmatrix}
	0 & \frac{\partial}{\partial x}+D_{Y\times \mathbb{R_+}} \\
	-\frac{\partial}{\partial x}+D_{Y\times \mathbb{R_+}} & 0
	\end{bmatrix}.
	\]
	Hence
	\begin{eqnarray*}
		&& \int_t^\infty  \tr_e  (e^{-sD_{Z_1\times \mathbb{R}_+}^-D_{Z_1\times \mathbb{R}_+}^+}D_{Z_1\times \mathbb{R}_+}^--e^{-sD_0^-D_0^+}D_0^- \chi_{Z_2\times \mathbb{R}_+} )\sigma \psi_{1,2}'   ds \\
		&=& \int_t^\infty  \tr_e  (e^{s\frac{\partial^2}{\partial^2 x}}(e^{-sD_{Z_1}^2}D_{Z_1}-e^{-sD_{Y\times \mathbb{R}_+}^2}D_{Y\times \mathbb{R}_+} ) \psi_{1,2}'   ds \\
		&&-\int_t^\infty  \tr_e  (e^{s\frac{\partial^2}{\partial^2 x}}\frac{\partial }{\partial x}(e^{-sD_{Z_1}^2}-e^{-sD_{Y\times \mathbb{R}_+}^2}) \psi_{1,2}'   ds \\
		&=& \int_t^\infty  \tr_e (e^{s\frac{\partial^2}{\partial^2 x}}  \psi_{2}') \tr_e (e^{-sD_{Z_1}^2}D_{Z_1}-e^{-sD_{Y\times \mathbb{R}_+}^2}D_{Y\times \mathbb{R}_+} )  ds \\
		&&- \int_t^\infty  \tr_e (e^{s\frac{\partial^2}{\partial^2 x}}\frac{\partial }{\partial x}  \psi_{2}')  \tr_e(e^{-sD_{Z_1}^2}-e^{-sD_{Y\times \mathbb{R}_+}^2})   ds.
	\end{eqnarray*}
	However, the kernel of $e^{s\frac{\partial^2}{\partial^2 x}}$ on $\mathbb{R}$ equals
	\[
	\kappa_{e^{s\frac{\partial^2}{\partial^2 x}}}(u,u'):=\frac{1}{\sqrt{4\pi s}}e^{\frac{-|u-u'|^2}{4s}},
	\]
	while the kernel of
	$e^{s\frac{\partial^2}{\partial^2 x}}\frac{\partial }{\partial x}$ equals
	\[
	\kappa_{e^{s\frac{\partial^2}{\partial^2 x}}\frac{\partial }{\partial x}}(u,u'):=-\frac{1}{\sqrt{4\pi s}}\frac{|u-u'|}{4s}e^{\frac{-|u-u'|^2}{4s}}.
	\]
	Hence
	\[
	\tr_e (e^{s\frac{\partial^2}{\partial^2 x}} \psi_{2}')
	= \int_{-\infty}^\infty   \kappa_{e^{s\frac{\partial^2}{\partial^2 x}}}(x,x) \psi_{2}'dx
	= \frac{1}{\sqrt{4\pi s}}
	\]
	and
	\[
	\tr_e (e^{s\frac{\partial^2}{\partial^2 x}}\frac{\partial }{\partial x} \psi_{2}')
	= \int_{-\infty}^\infty  \kappa_{e^{s\frac{\partial^2}{\partial^2 x}}\frac{\partial }{\partial x}}(x,x) \psi_{2}'dx
	= 0.
	\]
	This completes the proof.
	\epr
	
	\begin{Lem}\label{lem:second term}
		We have
		\begin{enumerate}
			\item
			$\tr_e (\sigma(\psi_{1,2}'\psi_{2,2}(1-\chi_{\psi_{2,2}=1})+\psi_{1,2}\psi_{2,2}'(1-\chi_{\psi_{1,2}=1}))(Q_0(t)-Q_0'(t)))=0$
			\item
			$\tr_e (\sigma(\psi_{1,2}'(\chi_{Z_1\times \mathbb{R}_+}-\chi_{\psi_{2,2}=1}) + \psi_{2,2}'(\chi_{Z_2\times \mathbb{R}_+}- \chi_{\psi_{1,2}=1}))(Q_0(t)-Q_0'(t)))=0$
		\end{enumerate}
	\end{Lem}
	\bpr
	
	Actually we have
	\begin{equation}\label{eq the first of 4}
		\tr_e(\sigma\psi_{1,2}'\psi_{2,2}(1-\chi_{\psi_{2,2}=1}) (Q_0(t)-Q_0'(t) ) )=0
	\end{equation}
	\begin{equation}\tr_e(\sigma\psi_{1,2}\psi_{2,2}'(1-\chi_{\psi_{1,2}=1}) (Q_0(t)-Q_0'(t)) )=0
	\end{equation}
	\begin{equation} \tr_e(\sigma\psi_{1,2}'(\chi_{Z_1\times \mathbb{R}_+}-\chi_{\psi_{2,2}=1} ) (Q_0(t)-Q_0'(t)) )=0   \end{equation}
	\begin{equation}
		\tr_e( \sigma\psi_{2,2}'(\chi_{Z_2\times \mathbb{R}_+}- \chi_{\psi_{1,2}=1})(Q_0(t)-Q_0'(t))  )=0
	\end{equation}
	We will prove Equation \eqref{eq the first of 4} only, the other three are totally parallel.
	
	%Note that both $\psi_{1,2}'\psi_{2,2}(1-\chi_{\psi_{2,2}=1})+\psi_{1,2}\psi_{2,2}'(1-\chi_{\psi_{1,2}=1})$ and $\psi_{1,2}'(\chi_{Z_2\times \mathbb{R}_+}-\chi_{\psi_{2,2}=1}) + \psi_{2,2}'(\chi_{Z_1\times \mathbb{R}_+}- \chi_{\psi_{1,2}=1})$ are supported on $ Y\times [0,100]\times [0,100]$.
	
	% Thus it is sufficient to show
	%\[
	%\tr_e (  (Q_0(t)-Q_0'(t) )\chi_{Y\times [0,100]\times [0,100] }  )=0.
	%\]
	Now, we have
	\[
	\tr_e (  (Q_0(t)-Q_0'(t))\sigma \psi_{1,2}'\psi_{2,2}(1-\chi_{\psi_{2,2}=1})  )=\int_t^{\infty}\tr_e  (e^{-sD_0^{-}D^+_0}D^-_0 \sigma \psi_{1,2}'\psi_{2,2}(1-\chi_{\psi_{2,2}=1}) )ds.
	\]
	By definition, we have
	\begin{eqnarray*}
		D_0^-&=&\begin{bmatrix}
			\frac{\partial}{\partial x} +i\frac{\partial }{\partial y} & D_Y^{-} \\
			-D_Y^{+} & -\frac{\partial}{\partial x} +i\frac{\partial }{\partial y}
		\end{bmatrix}\\
		D_0^+&=& \begin{bmatrix}
			-\frac{\partial}{\partial x} +i\frac{\partial }{\partial y} & -D_Y^-\\
			D_Y^+ & \frac{\partial}{\partial x} +i\frac{\partial }{\partial y}
		\end{bmatrix}.
	\end{eqnarray*}
Direct computation shows
	\begin{eqnarray*}
		&&\tr_e  (e^{-sD_0^{-}D^+_0}D^-_0 \sigma \psi_{1,2}\psi_{2,2}'(1-\chi_{\psi_{1,2}=1}) )\\
		&=& \tr_e(e^{-s\frac{\partial^2}{\partial^2 x}} \frac{\partial }{\partial x}   \psi_2' )\tr_e(e^{-s\frac{\partial^2}{\partial^2 y} } \sigma \psi_2 (1-\chi_{\psi_2=1}) )\tr_e ( \begin{bmatrix}e^{-s D_Y^- D_Y^+} &0 \\ 0 & e^{-s D_Y^+ D_Y^-}  \end{bmatrix} )\\
		&&+ \tr_e(e^{-s\frac{\partial^2}{\partial^2 x} } \psi_2')\tr_e(e^{-s\frac{\partial^2}{\partial^2 y} } \frac{\partial}{\partial y}  \psi_2 (1-\chi_{\psi_2=1}))\tr_e ( \begin{bmatrix}e^{-s D_Y^- D_Y^+} &0 \\ 0 & e^{-s D_Y^+ D_Y^-}  \end{bmatrix} )\\
		&&+ \tr_e(e^{-s\frac{\partial^2}{\partial^2 x} } \psi_2')\tr_e(e^{-s\frac{\partial^2}{\partial^2 y} } \psi_2 (1-\chi_{\psi_2=1})  )\tr_e( \begin{bmatrix}0 & e^{-s D_Y^- D_Y^+} D^{-}_Y\\   - e^{-s D_Y^+ D_Y^-}D_Y^+ &0 \end{bmatrix} )
	\end{eqnarray*}
	In the proof of Lemma \ref{lem:first term}, we have shown that
	\[
	\tr_e(e^{-s\frac{\partial^2}{\partial^2 x}} \frac{\partial}{\partial x}  \psi_2'  )=\tr_e(e^{-s\frac{\partial^2}{\partial^2 y} } \frac{\partial}{\partial y} \psi_2 (1-\chi_{\psi_2=1}))=0.
	\]
	At the same time, we have
	\[
	\tr_e ( \begin{bmatrix}0 & e^{-s D_Y^- D_Y^+} D^{-}_Y\\   - e^{-s D_Y^+ D_Y^-}D_Y^+ &0 \end{bmatrix} )=0.
	\]
	The proof is completed.
	\epr
	
	In a word, by the above argument and Lemmas \ref{lem:first term} and \ref{lem:second term}, we have
	\begin{Thm}\label{thm:muller thm}
		The following equality holds:
		\begin{eqnarray*}
			\tr_e (\ind_G D_M)& =& \int_{M} I(e)
			- \lim_{t\to 0} \frac{1}{\sqrt{4\pi }} \int_t^\infty \tr_e ( e^{-sD_{Z_1}^2} D_{Z_1} - e^{-sD_{Y\times \mathbb{R}_+}^2} D_{Y\times \mathbb{R}_+}) \frac{1}{\sqrt{s}} ds \\
			&&- \lim_{t\to 0}\frac{1}{\sqrt{4\pi }} \int_t^\infty \tr_e ( e^{-sD_{Z_2}^2} D_{Z_2} - e^{-sD_{Y\times \mathbb{R}_+}^2} D_{Y\times \mathbb{R}_+}) \frac{1}{\sqrt{s}} ds.
		\end{eqnarray*}
	\end{Thm}
	
	%\begin{Cor}
	%We have
	%\[
	%\eta_{\infty,g}(D_X)= \eta_{\partial_1 M, Y,g}(D_{\partial_1 M}) + \eta_{\partial_2 M, Y,g}(D_{\partial_2 M})
	%\]
	%\end{Cor}

\bibliography{invertible_at_infinity}

\end{document}